\documentclass{amsart}
\usepackage{amssymb}
\usepackage{graphicx}
\usepackage{epstopdf}
\usepackage{subfig}
\usepackage{float}
\usepackage{placeins}

\theoremstyle{plain}
\newtheorem{theorem}{Theorem}[section]
\newtheorem{conjecture}[theorem]{Conjecture}

\newtheorem{lemma}[theorem]{Lemma}
\newtheorem{corollary}[theorem]{Corollary}

\newtheorem{claim}[theorem]{Claim}
\newtheorem{problem}[theorem]{Problem}
\theoremstyle{definition}
\newtheorem{definition}[theorem]{Definition}
\theoremstyle{remark}

\newtheorem{example}[theorem]{Example}

\def\U{\mathbb{U}}

\def\R{\mathbb{R}}
\def\C{\mathbb{C}}
\def\I{\mathcal{I}}
\def\J{\mathcal{J}}
\def\Y{\mathcal{Y}}

\title[]{Arrangements Of Minors In The Positive Grassmannian And a Triangulation of The Hypersimplex}
\author{Miriam Farber \and Yelena Mandelshtam}
\address{Department of Mathematics, Massachusetts Institute of Technology,
77 Massachusetts Avenue, Cambridge MA 02139}
\email{mfarber@mit.edu}
\address{Department of Mathematics, Stanford University}
\email{yelena13@stanford.edu}
\thanks{M.F.\ is supported in part by the NSF graduate research fellowship grant 1122374}


\begin{document}

\begin{abstract}

The structure of zero and nonzero minors in the Grassmannian leads to
rich combinatorics of matroids. In this paper, we investigate an even
richer structure of possible equalities and inequalities between the
minors in the positive Grassmannian. It was previously shown that
arrangements of equal minors of largest value are in bijection with
the simplices in a certain triangulation of the hypersimplex that was
studied by Stanley, Sturmfels, Lam and Postnikov. Here we investigate
the entire set of arrangements and its relations with this
triangulation. First, we show that second largest minors correspond to
the facets of the simplices. We then introduce the notion of cubical
distance on the dual graph of the triangulation, and study its
relations with the arrangement of t-th largest minors. Finally, we show that arrangements of largest minors induce a structure of partially ordered sets on the entire collection of minors. We use the Lam and Postnikov circuit triangulation of the hypersimplex to describe a 2-dimensional grid structure of this poset.\\
\end{abstract}



\maketitle

\section{Introduction}
In this paper, we study the relations between equalities and inequalities of minors in the positive
Grassmannian and the triangulation of the hypersimplex.
This study is strongly tied to various combinatorial objects such as the positive Grassmannian and its stratification \cite{pos2}, alcoved polytopes\cite{polytopes}, sorted sets and Gr\"obner bases \cite{Stu}, as well as many other objects in algebraic combinatorics and beyond.

The notion of total positivity was originally introduced by Schoenberg \cite{sch} and Gantmacher and Krein \cite{gk} in the 1930s. The classical theory of total positivity deals with totally positive matrices- matrices in which all minors of all orders are positive. Later, the theory was extended by Lusztig in the general Lie theoretic setup through definition of the positive part for a reductive Lie group $G$ and a generalized partial flag manifold $G/P$. In \cite{pos2} it was shown that the space of totally positive matrices can be embedded in the positive Grassmannian, and this embedding unveils symmetries which are hidden on the level of matrices. Thus it is very natural to discuss equalities and inequalities of minors in the more general settings of the positive Grassmannian.

The number and positioning of equal minors in totally positive matrices was studied in
several recent papers. In \cite{FFJM}, it was shown that the number of
equal entries in a totally positive $n\times n$ matrix is $O(n^{4/3})$. The authors
also discussed positioning of equal entries and
obtained relations to the Bruhat order of permutations. In \cite{FRS} it
was shown, using incidences between points and hyperplanes, that the
maximal number of equal $k\times k$ minors in a $k\times n$ totally positive matrix is $O(n^{k-{k\over k+1}})$.

Inequalities between products of two minors in TP matrices have been
widely studied as well \cite{Ska,RS}, and have close ties with Temperley-Lieb
Immanants. Recently there has been also a study of products of
three minors in such matrices \cite{lam}, that related such products with
dimers. Despite all of the above, not much is known about the
inequalities between the minors themselves. What is the full structure
of all the possible equalities and inequalities between minors in TP matrices? The only
part of this problem that has been solved discusses the structure of
the minors with largest value and smallest value \cite{main}, while the
rest of the problem remains open. The description in \cite{main} involves rich combinatorial structure that relates arrangements of smallest minors with triangulations of the $n$-gon and the notion of weakly separated sets, while the structure of largest minors was related to thrackles and sorted collections. In this paper, we discuss the general case, and its tight relation with the triangulation of the hypersimplex.
\section{Background}
For $n\geq k\geq 0$, let the {\it Grassmannian} $Gr(k,n)$ (over $\R$) be the manifold
of $k$-dimensional subspaces $V \subset \R^n$. It can be identified with the space of real $k\times n$ matrices of rank $k$ modulo row operations. Here we
assume that the subspace $V$ associated with a $k\times n$-matrix $A$ is spanned by the row vectors of $A$.
For such a matrix $A$ and a $k$-element subset $I \subset [n]:=\{1,2,3\ldots,n\}$, we denote by $A_I$ the $k\times k$-submatrix of $A$ in the column set $I$, and let $\Delta_I(A):=\det(A_I)$. The coordinates $\Delta_I$ form projective coordinates on the Grassmannian, called
the {\it Pl\"ucker coordinates.} In \cite{pos2}, the {\it positive (nonnegative) Grassmannian\/} $Gr^+(k,n)$ ($Gr^{\geq}(k,n)$) was defined to be the subset of $Gr(k,n)$ whose elements are represented by $k\times n$ matrices $A$ with strictly positive (nonnegative) Pl\"ucker coordinates: $\Delta_I >0$ for all $I$.

We recall two classical stratifications of $Gr(k,n)$ \cite{pos2}. The first one is the cellular decomposition of $Gr(k,n)$ into a disjoint union of Schubert cells. The Grassmannian $Gr(k,n)$ also has a subdivision into matroid strata (or Gelfand-Serganova strata) $S_M$ labelled by matroids $M$:\\
Let $M \subset {[n]\choose k}$, and define
$$S_M:=\{A \in Gr^{\geq }(k,n) | \Delta_I(A)>0 \textrm{ iff } I \in M\}.$$
If $S_M \neq \emptyset$ then $M$ must be a matroid, and in such case $M$ is called positroid and $S_M$ is called a positroid cell. The nonnegative Grassmannian can be decomposed into cells via the positroid stratification $Gr^{\geq}(k,n)=\cup_M S_M$. This decomposition has been studied by Postnikov in \cite{pos2} and was described in terms of various combinatorial objects such as: decorated permutations, plabic graphs, Le-diagrams, Grassmann necklaces, etc. Strictly speaking, positroid cells correspond to arrangements
of zero and positive Pl\"ucker coordinates.\\
The following stratification, which is finer than the positroid stratification, was introduced in \cite{main}.
In this stratification, the strata are defined by all possible equalities and inequalities between
the Pl\"ucker coordinates.

\begin{definition}
Let $\U=(\U_0,\U_1, \dots,\U_l)$ be an ordered set-partition of the set $[n]\choose k$ of all $k$-element subsets in $[n]$.
Let us subdivide the nonnegative Grassmannian $Gr^{\geq}(k,n)$ into the strata $S_\U$ labelled by such ordered set partitions
$\U$ and given by the conditions:
\begin{enumerate}
\item
$\Delta_I = 0$ for $I\in \U_0$,
\item
$\Delta_I = \Delta_J$ if $I,J\in \U_i$,
\item
$\Delta_I < \Delta_J$ if $I\in \U_i$ and $J\in \U_j$ with $i<j$.
\end{enumerate}
An {\it arrangement of minors\/} is an ordered set-partition $\U$ such that the stratum $S_\U$ is not empty.
\end{definition}
The problem bellow was suggested in \cite{main}:
\begin{problem}
Describe combinatorially all possible arrangements of minors in $Gr^{\geq}(k,n)$.
Investigate the geometric and the combinatorial structure of the stratification
$Gr^{\geq}(k,n)=\bigcup S_\U$.
\end{problem}
\begin{example}\label{example1}
Let
$$\U_0= \emptyset, \U_1=\Big\{\{3,4\}\Big\}, \U_2=\Big\{\{1,4\}\Big\}, \U_3=\Big\{\{1,2\},\{2,3\},\{1,3\},\{2,4\}\Big\}.$$
Then $\U=(\U_0,\U_1,\U_2,\U_3)$ is an ordered set partition of $[4]\choose 2$. Consider the matrix $A=\left(
                                                            \begin{array}{cccc}
                                                              1 & 2 & 1 & 1/3 \\
                                                              1 & 3 & 2 & 1 \\
                                                            \end{array}
                                                          \right)$, which satisfies $$\Delta_{34}=1/3, \Delta_{14}= 2/3, \Delta_{12}=\Delta_{23}=\Delta_{13}=\Delta_{24}=1.$$ Therefore $S_\U$ is nonempty since $A \in S_\U$, and hence $\U$ is an arrangement of minors.
\end{example}
For the case $k=1$, the stratification of $Gr^{\geq}(k,n)$ into the strata $S_\U$ is equivalent to \emph{Coxeter arrangement} of type A (also known as \emph{braid arrangement}). The classification of the possible options for $U_0$ is equivalent to the positroid stratification described above. In this work we deal with the positive Grassmannian, and thus restrict ourself to the case $\U_0 = \emptyset$. We extend the convention from \cite{main}:
\begin{definition}
We say that a subset $\mathcal{J}\subset{[n]\choose k}$ is
an {\it arrangement of $t^{th}$ largest (smallest) minors\/} in $Gr^{+}(k,n)$, if there exists a nonempty stratum $S_{\U}$ such that $\U_0=\emptyset$ and $\U_{l-t+1}=\mathcal{J}$ ($\U_{t}=\mathcal{J}$).

If $t=1$ we say that such arrangement is the arrangement of largest (smallest) minors.
\end{definition}
Arrangements of largest and smallest minors were studied in
\cite{main}, where it was shown that they enjoy a rich combinatorial
structure. Arrangements of smallest minors are related to weakly
separated sets. Such sets were originally introduced by Leclerc-Zelevinsky \cite{LZ} in the study of quasi-commuting quantum minors, and are closely related to the associated {\it cluster algebra} of the positive Grassmannian. Arrangement of largest minors were shown to be in bijection with simplices of Sturmfels' triangulation of the hypersimplex, which also appear in the context of Gr\"obner bases \cite{polytopes}.
In this paper, we are interested in the combinatorial description of  arrangements of $t^{th}$ largest minors for $t \geq 2$. For a stratum $S_{\U}$, the structure of $\U_t$ for $t<l$ depends on the structure of $\U_l$, as we will show later.
\begin{definition}
Let $\mathcal{J}\subset{[n]\choose k}$ be an arrangement of largest minors. We say that $\mathcal{Y}\subset{[n]\choose k}$ is a $(t,\mathcal{J})-$largest arrangement ($t \geq 2$) if
there exists a nonempty stratum $S_{\U}$ such that $\U_0=\emptyset$, $\U_{l}=\mathcal{J}$ and $\U_{l-t+1}=\mathcal{Y}$.

We say that $W \in {[n]\choose k}$ is a $(t,\mathcal{J})-$largest minor if there exists a $(t,\mathcal{J})-$largest arrangement $\mathcal{Y}$ such that $W \in \mathcal{Y}$
\end{definition}

In particular, if $\mathcal{Y}\subset{[n]\choose k}$ is a $(t,\mathcal{J})-$largest arrangement, then $\mathcal{Y}$ is also an arrangement of $t^{th}$ largest minors. Example~\ref{example1} implies that $\Big\{\{3,4\}\Big\}$ is a $(3,\Big\{\{1,2\},\{2,3\},\{1,3\},\{2,4\}\Big\})-$largest arrangement, and that $\{1,4\}$ is a\\ $(2,\Big\{\{1,2\},\{2,3\},\{1,3\},\{2,4\}\Big\})-$largest minor.

\section{The Triangulation of the Hypersimplex}
\label{sec:Triangulations}
\begin{definition}
The hypersimplex $\Delta_{k,n}$ is an $(n-1)$-dimensional polytope defined as follows:
\begin{center}
$\Delta_{k,n}=\{(x_1,\ldots,x_n)\textrm{ }|\textrm{ }0 \leq x_1 ,\ldots,x_n \leq 1;
x_1+x_2+\ldots+x_n = k\}$.
\end{center}
\end{definition}
Laplace showed that the normalized volume of $\Delta_{k,n}$
equals the Eulerian number $A(n-1,k-1)$, that is, the number of permutations $w$ of size $n-1$ with exactly $k-1$ descents. A bijective proof of this property was given by Stanley in \cite{Sta1}. In \cite{polytopes} four different constructions of a triangulation of the hypersimplex into $A(n-1,k-1)$ unit simplices
are presented: Stanley's triangulation \cite{Sta1}, Alcove triangulation, circuit triangulation and Sturmfels'
triangulation \cite{Stu}. It was shown in \cite{polytopes} that
these four triangulations coincide. We now describe Sturmfels' construction
following the notations of \cite{polytopes}. Afterwards we describe
the circuit triangulation as it appears in \cite{polytopes}.
\subsection{Sturmfels' construction}
\begin{definition}
For a multiset $S$ of elements from $[n]$, let $Sort(S)$ be the non-decreasing sequence obtained by
ordering the elements of $S$. Let $I,J \subset {[n]\choose k}$ and let
$Sort(I \cup J)=(a_1,a_2,\ldots,a_{2k})$. Define
\begin{center}
$Sort_1(I,J):=\{a_1,a_3,\ldots,a_{2k-1}\}$, $Sort_2(I,J):=\{a_2,a_4,\ldots,a_{2k}\}$.
\end{center}
A pair $\{I,J\}$ is called {\it
sorted\/} if $Sort_1(I, J)=I$ and $Sort_2(I, J)=J$, or vice versa.
\end{definition}
For example, $\{1,3,5\}, \{2,4,6\}$ are sorted, while $\{1,4,5\}, \{2,3,6\}$ are not sorted.
We would like to mention a useful property of sortedness, which follows from Skandera inequalities \cite{Ska} (see also Theorem 6.3 in \cite{main}).
\begin{corollary}\label{cor:Skandera}
Let $I,J\in{[n]\choose k}$ be a pair which is not sorted. Then\\ $ \Delta_{sort_1(I,J)} \, \Delta_{sort_2(I, J)} > \Delta_I \Delta_J$ for points of the positive Grassmannian $Gr^+(k,n)$.
\end{corollary}
\begin{definition}
A collection $\I=\{I_1,I_2,\ldots,I_r\}$ of
elements in $[n]\choose k$ is called sorted if $I_i,I_j$ are sorted, for
any pair $1 \leq i < j \leq n$. Equivalently, if $I_i=\{{a_1}^i<{a_2}^i<\ldots<{a_k}^i\}$ for all $i$
then $\I$ is sorted if (after possible reordering of the $I_i$'s) we have
$${a_1}^1 \leq {a_1}^2 \leq \ldots \leq {a_1}^k \leq {a_2}^1 \leq {a_2}^2 \leq \ldots {a_2}^k \leq \ldots \leq {a_r}^1\leq {a_r}^2 \ldots \leq {a_r}^k$$
\end{definition}
Given $I \in {[n]\choose k}$, let $\epsilon_I$ be a 0,1-vector
$\epsilon_I=(\epsilon_1,\epsilon_2,\ldots,\epsilon_n)$ such that
$\epsilon_i=1$ iff $i\in I$, and otherwise $\epsilon_i=0$. In some cases we will use $I$ instead of $\epsilon_I$ (if it is clear from the context). For a sorted collection $\I$, we denote by
$\nabla_\I$ the $(r-1)$-dimensional simplex with the vertices
$\epsilon_{I_1},\ldots,\epsilon_{I_r}$.
\begin{theorem}\cite{Stu}\label{2.1}
The collection of simplices $\nabla_\I$ where $\I$ varies over
all sorted collections of $k$-element subsets in $[n]$ , is a simplicial
complex that forms a triangulation of the hypersimplex $\Delta_{k,n}$.
\end{theorem}
From Theorem \ref{2.1}, it follows that the maximal by inclusion sorted collections correspond to the maximal simplices in the triangulation, and they are known to be of size $n$.

As an example, consider the case $k=2$. Let $I=\{a,b\},J=\{c,d\} \subset {[n]\choose 2}$ be a pair of sorted sets ($I \neq J$). Consider the graph $G$ of order $n$ whose vertices lie in clockwise order on a circle. Then we can think about $I$ and $J$ as edges in the graph, and since $I$ and $J$ are sorted, these two edges either share a common vertex or cross each other.
\begin{definition}
A {\it thrackle} is a graph in which every pair of edges is either crossing or shares a common vertex.
\footnote{Our thrackles are a special case of Conway's thrackles.
The latter are not required to have vertices arranged on a circle.}
\end{definition}
The maximal number of edges in a thrackle is $n$, and each such maximal thrackle corresponds to a maximal sorted set with $k=2$. Figure \ref{fig:more_thrackles} describes all the thrackles of order up to 5.
\begin{figure}[h]
\hspace{0.7in}
\includegraphics[height=0.6in,width=.6in]{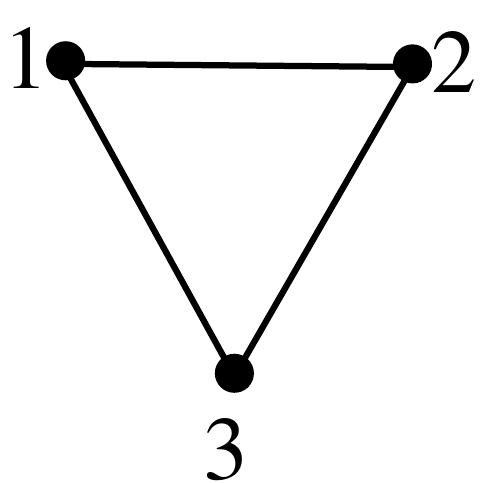}%
\qquad\quad
\includegraphics[height=0.6in,width=.6in]{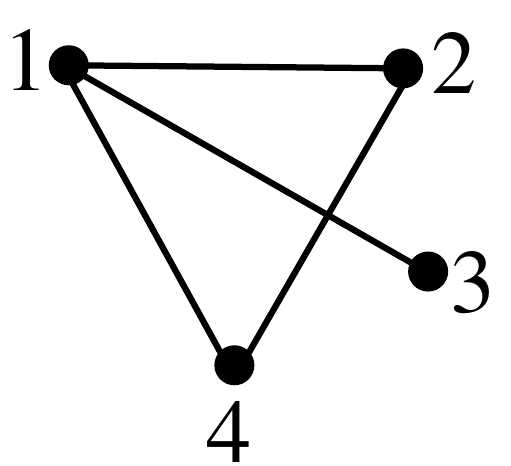}%
\qquad
\includegraphics[height=0.6in,width=.6in]{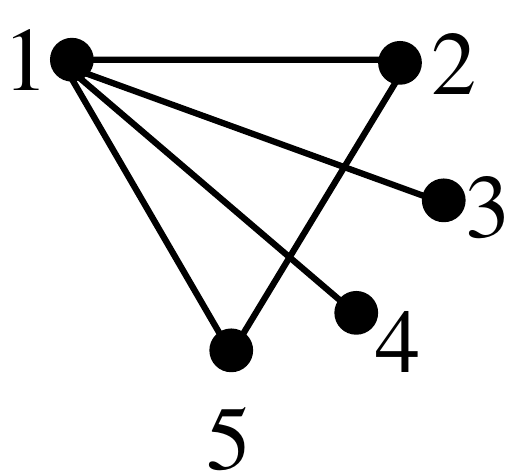}%
\qquad
\includegraphics[height=0.6in,width=.6in]{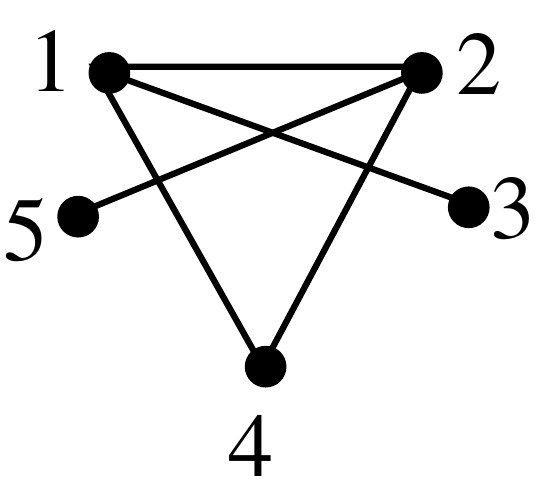}%
\qquad
\includegraphics[height=0.6in,width=.6in]{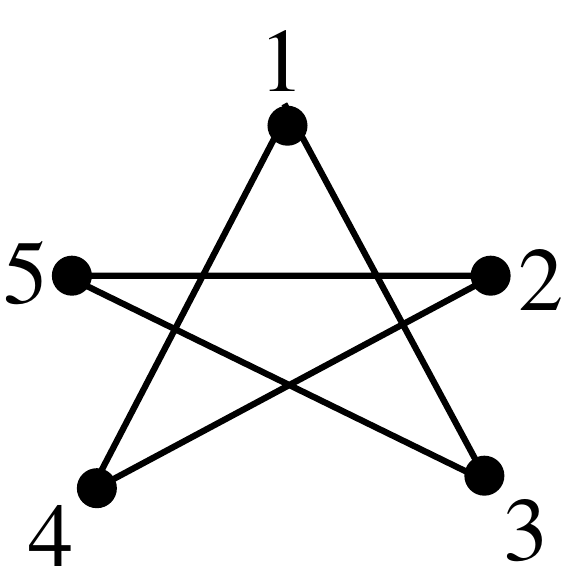}%
\caption{All maximal thrackles that have at most 5 vertices (up to rotations and reflections).}
\label{fig:more_thrackles}
\end{figure}
\begin{definition}
The \emph{dual graph} $\Gamma_{(k,n)}$ of Sturmfels' triangulation of $\Delta_{k,n}$ is the graph whose
vertices are the maximal simplices, and two maximal simplices are adjacent by an edge if they share a common facet.
\end{definition}

Figure~\ref{graphGamma2n} depicts the graph $\Gamma_{(2,6)}$. This graph has $A(5,1)=26$ vertices, each corresponds to a maximal thrackle on 6 vertices. We also described explicitly 6 of the vertices. In particular, vertices $a$ and $b$ are connected since $b$ can be obtained from $a$ by removing the edge $\{1,6\}$ and adding instead the edge $\{2,5\}$. Therefore $\nabla_{a}$ and $\nabla_{b}$ share a common facet.
\begin{figure}[h]
\centering
\includegraphics[height=2.5in,width=3.8in]{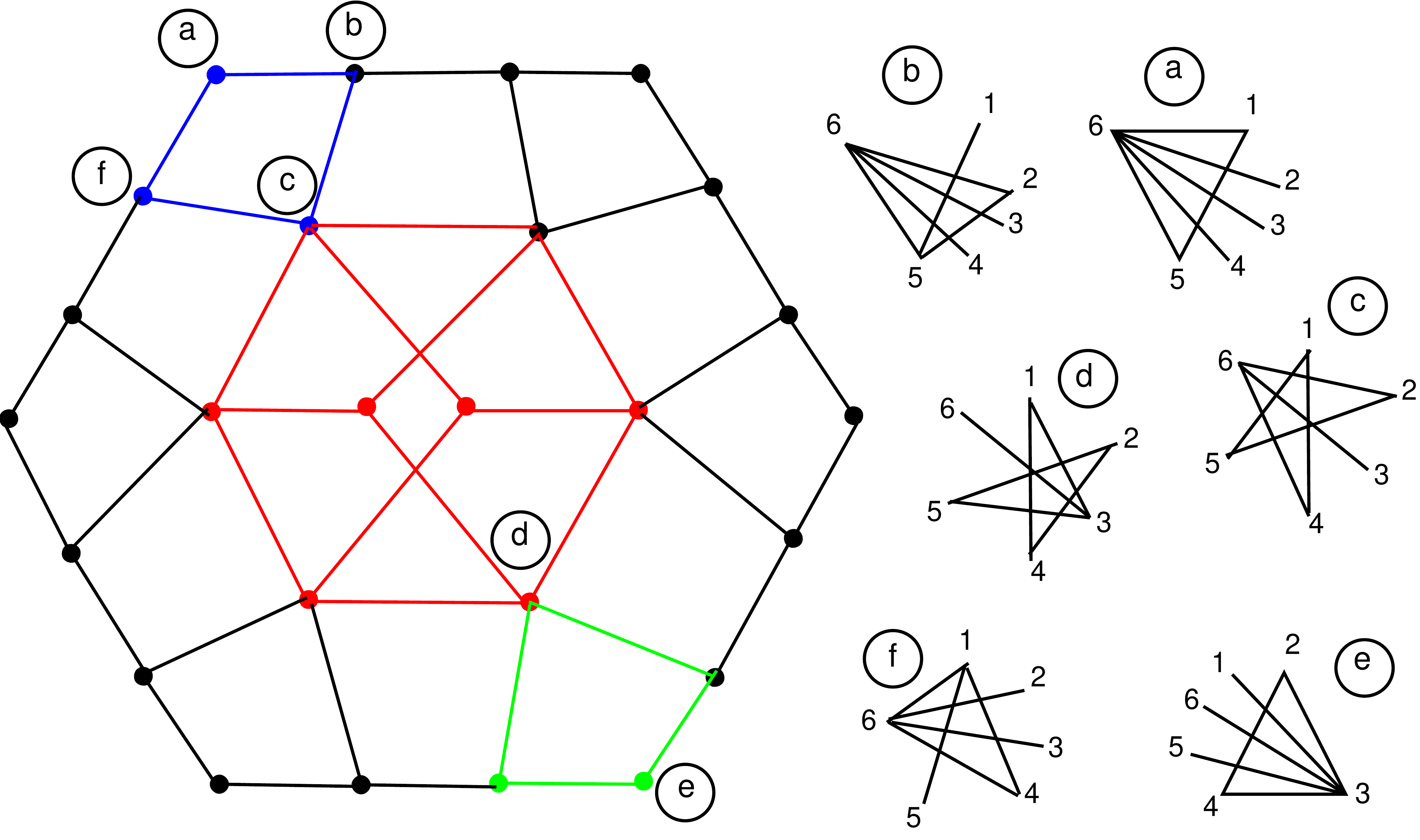}
\caption{The graph $\Gamma_{(2,6)}$}
\label{graphGamma2n}
\end{figure}

\subsection{Circuit triangulation}
We start by defining the graphs $G_{k,n}$ and circuits in these graphs. These definitions are taken from \cite{polytopes}.
\begin{definition}
We define $G_{k,n}$ to be the directed graph whose vertices are
$\{\epsilon_I\}_{I \in {[n]\choose k}}$, and two vertices
$\epsilon=(\epsilon_1,\epsilon_2,\ldots,\epsilon_n)$ and $\epsilon'$ are
connected by an edge oriented from $\epsilon$ to $\epsilon'$ if there
exists some $i \in [n]$ such that $(\epsilon_{i},\epsilon_{i+1})=(1,0)$
and the vector $\epsilon'$ is obtained from $\epsilon$ by switching
$\epsilon_{i},\epsilon_{i+1}$ (and leaving all the other coordinates
unchanged, so the 1 is ``shifted'' one place to the right). We give
such an edge the label, $i$. When
considering $i \in [n]$ we refer to $i$ as $i \textrm{ mod } n$, and thus if
$i=n$, we have $i+1=1$.
\end{definition}
A circuit in $G_{k,n}$ of minimal possible length must be of length $n$, and is given by a sequence
of shifts of ``1''s: The first ``1'' in $\epsilon$ moves to the position of the second ``1'',
the second ``1'' moves to the position of the third ``1'', and so on, finally, the last
``1'' cyclically moves to the position of the first ``1''. Figure~\ref{circuit38} is an example of a minimal circuit in $G_{3,8}$.
For convenience, we label the vertices by $I$ instead of $\epsilon_I$.
The sequence of labels of edges in a minimal circuit forms a permutation
$\omega=\omega_1\omega_2\ldots \omega_n \in S_n$, and two permutations
that are obtained from each other by
cyclic shifts correspond to the same circuit. Thus, we can label each minimal circuit in $G_{k,n}$ by its
permutation modulo cyclic shifts. For example, the permutation corresponding to the minimal circuit in
Figure~\ref{circuit38} is $\omega= 56178243$, and we label this circuit $C_\omega$.
\begin{figure}[h]
\centering
\includegraphics[height=1.3in]{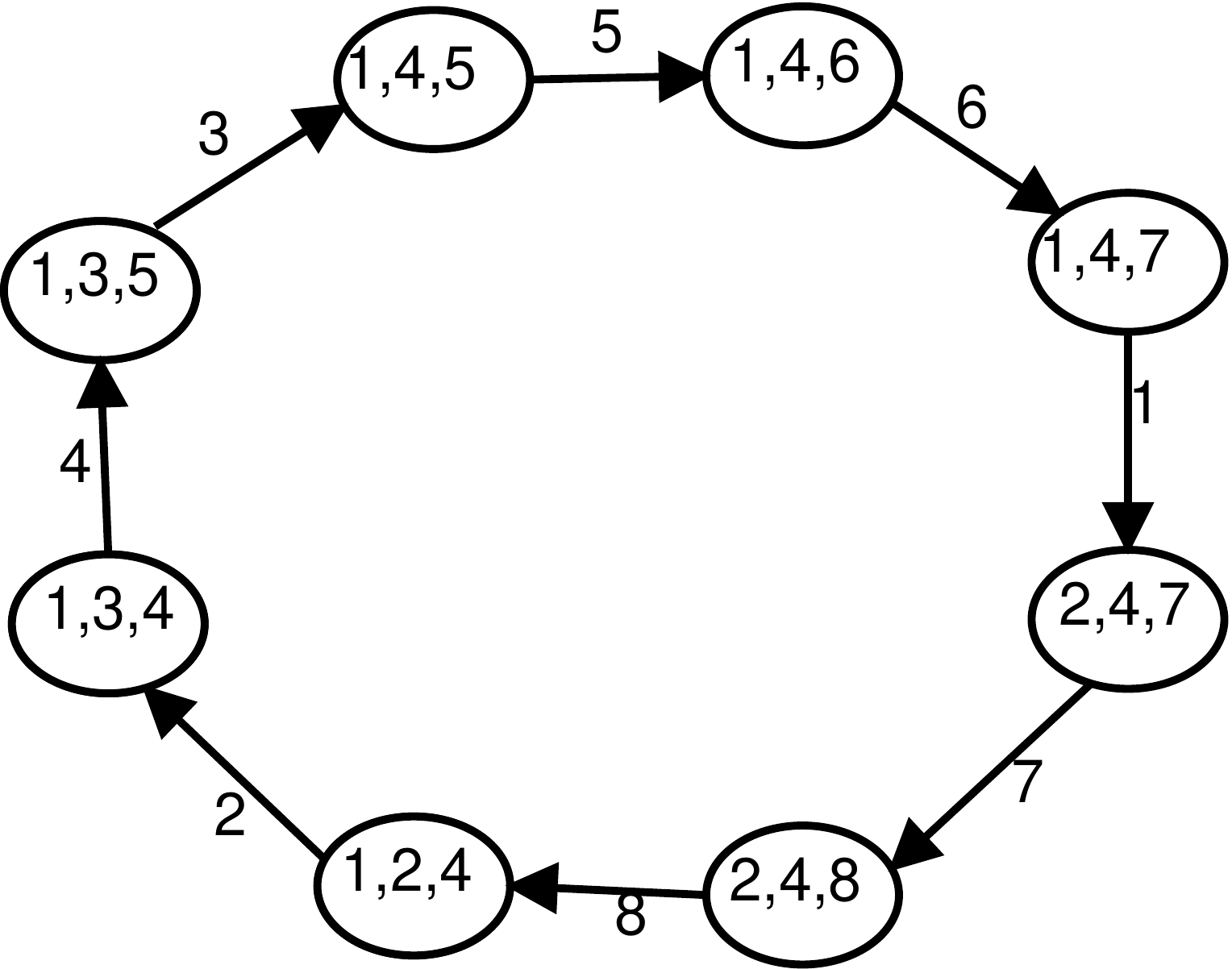}%
\caption{A minimal circuit in $G_{3,8}$}
\label{circuit38}
\end{figure}
Circuit triangulation is described in the following theorem.
\begin{theorem}\cite{polytopes}
Each minimal circuit $C_\omega$ in $G_{k,n}$ determines the simplex $\Delta_\omega$ inside the hypersimplex $\Delta_{k,n}$ with the vertex set $C_\omega$. The collection of simplices $\Delta_\omega$ corresponding to all minimal circuits in $G_{k,n}$ forms a triangulation of the hypersimplex, which is called the circuit triangulation. The vertices of $C_\omega$ form a maximal sorted collection, and every maximal sorted collection can be realized via a minimal circuit in the graph $G_{k,n}$.
\label{circuitthm}
\end{theorem}

Circuit triangulation proves to be a useful tool when studying adjacency of maximal simplices in the hypersimplex, and understanding the structure of $\Gamma_{(k,n)}$. In particular, the following theorem implies that the maximal degree of a vertex in $\Gamma_{(k,n)}$ is at most $n$.
\begin{theorem}\cite{polytopes}
Let $\mathcal{I} = \{I_1,I_2,\ldots,I_n\}$ be a sorted subset corresponding to the maximal
simplex $\nabla_\I$ of $\Gamma_{(k,n)}$. Let $t \in [n]$ and $I_t = \{i_1,i_2,\ldots,i_k\}$. Then we can replace
$I_t$ in $\I$ by another $I'_t \in {[n] \choose k}$ to obtain an adjacent maximal simplex $\nabla_{\I'}$ if and
only if the following holds: We must have $I'_t = \{i_1,\ldots,i'_a,\ldots,i'_b,\ldots,i_k\}$ for some $a \neq b \in [n]$ and $i'_a \neq i'_b$, $i_a-i'_a=i'_b-i_b=\pm 1 (mod \textrm{ } n)$ and also both
$k$-subsets $I_{c}=\{i_1,\ldots,i'_a,\ldots,i_b,\ldots,i_k\}$ and $I_{d}=\{i_1,\ldots,i_a,\ldots,i'_b,\ldots,i_k\}$ must lie in $\I$.
\label{thm:adjacency}
\end{theorem}

In terms of minimal circuits, $I'_t$ is obtained by a \emph{detour} from the minimal circuit that
corresponds to $\mathcal{I}$, as presented in Figure~\ref{detour}. Every detour can be defined by the triple
$\{I_c,I_t,I_d\}$ (again see Figure~\ref{detour}).

\begin{figure}[h]
\centering
\includegraphics[height=1.3in]{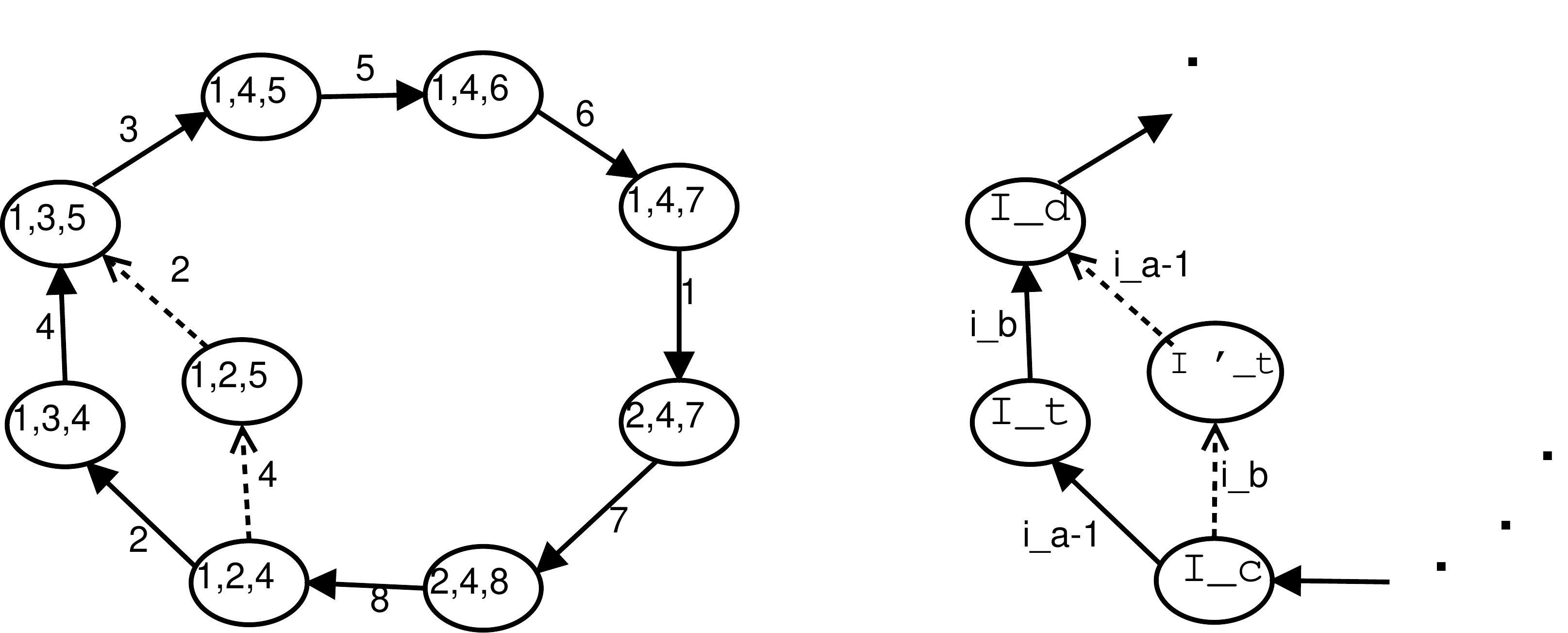}%
\caption{The figure on the left is a minimal circuit in $G_{3,8}$. The tuple (1,3,4) can be replaced with the tuple (1,2,5) according to Theorem~\ref{thm:adjacency}. The figure on the right depicts the situation described in the theorem.}
\label{detour}
\end{figure}
\section{Arrangements of second largest minors}
In this section, we describe necessary and sufficient conditions on
arrangements of second largest minors. Theorem~\ref{thm:sorted} (given
below) implies that maximal arrangements of largest minors are in
bijection with the vertices of $\Gamma_{(k,n)}$. In this section, we
will show that the structure of arrangements of second largest minors
is strongly related to the structure of edges in $\Gamma_{(k,n)}$. Then, in the next section, we discuss necessary conditions for arrangements of $t$-th largest minors for any $t \geq 2$. The case $t=1$, i.e., arrangements of largest minors, was fully resolved in \cite{main}:
\begin{theorem}
A subset of $\mathcal{J} \subset {[n]\choose k}$ is an arrangement of largest minors in $Gr^+(k,n)$ if and only if it is a sorted subset. Equivalently, $\mathcal{J}$ is an arrangement of largest minors if and only if it corresponds to a simplex in Sturmfels' triangulation of the hypersimplex. Maximal arrangements of largest minors contain exactly $n$ minors. The number of maximal arrangements of largest minors in $Gr^+(k,n)$ equals the Eulerian number $A(n-1,k-1)$.
\label{thm:sorted}
\end{theorem}
For completeness, let us also mention the known result regarding arrangements of smallest minors. We first introduce the concept of weak separation as it appears in \cite{LZ}.

\begin{definition}
Let $I,J\in {[n]\choose k}$, and denote $I\setminus J = \{a_1,\dots,a_r\}$, $J\setminus I=\{b_1,\dots,b_r\}$ for $a_1 < \ldots < a_r$ and $b_1 < \ldots < b_r$. We say that $I$ and $J$ are weakly separated if there exists some $ 0 \leq s \leq r$ such that
$$a_1 < \cdots < a_s < b_1< \dots < b_r < a_{s+1} < \cdots < a_r,$$
or
$$b_1 < \cdots < b_s < a_1< \dots < a_r < b_{s+1} < \cdots < b_r.$$

A subset of $[n]\choose k$ is called weakly separated if every two elements in it are weakly separated.
\end{definition}

In \cite{LZ}, Leclerc and Zelevinsky conjectured that all maximal (by containment) weakly
separated subsets in ${[n]\choose k}$ have the same number of elements, which equals $k(n-k)+1$.
This conjecture was proved independently in
\cite{DKK} and in \cite{OPS}.

\begin{theorem}\cite{main}
Any weakly separated subset in ${[n]\choose k}$ is an arrangement of smallest minors in $Gr^+(k,n)$. If $k \in \{1,2,3, n-3,n-2,n-1\}$, then the converse is also true. Maximal (by size) arrangement of smallest minors contains at least $k(n-k)+1$ elements.
\label{weakseparation}
\end{theorem}

As a warm-up, we start our discussion with the case $k=2$.
\subsection{The case $k=2$: maximal thrackles}
Consider the space $Gr^{+}(2,n)$, and let
$\mathcal{J}\subset{[n]\choose 2}$ be a maximal arrangement of largest
minors (hence it corresponds to a maximal thrackle. We will later
consider the case in which no maximality assumption is involved). Given $W \in {[n]\choose 2}$, we ask whether $W$ is $(2,\mathcal{J})-$largest minor. That is, whether there exists an element in $Gr^{+}(2,n)$ in which the collection of largest minors is $\mathcal{J}$ and $W$ is second largest. Our theorem below gives necessary and sufficient conditions on such $W$.
\begin{theorem}
Let $W \in {[n]\choose 2}$ and let $\mathcal{J}\subset{[n]\choose 2}$ be some maximal arrangement of largest minors, such that $W \notin \mathcal{J}$. The following four statements are equivalent.
\begin{enumerate}
  \item $W$ is a $(2,\mathcal{J})-$largest minor.
  \item There exists a vertex $\mathcal{Q}$ in $\Gamma_{(2,n)}$ that is adjacent to $\mathcal{J}$, such that $W \in \mathcal{Q}$.
  \item There exists $J \in \mathcal{J}$ such that $(\mathcal{J} \setminus J) \cup W$ is an arrangement of largest minors.
  \item There exist four distinct vertices labelled $a, a+1, b, b+1 \textrm{ (mod) } n$ such that
$\{(a,b),(a-1,b),(b+1,a)\} \subset \mathcal{J}$ and $W=(a-1, b+1)$.
\end{enumerate}
In particular, the minors that can be second largest are in bijection with the edges of $\Gamma_{(2,n)}$ that are connected to vertex $\mathcal{J}$, and the number of such minors is at most $n$.
\label{thm:2byn}
\end{theorem}
Theorems~\ref{thm:adjacency} and~\ref{thm:sorted} imply the equivalence $(2) \iff (3)\iff (4)$. The equivalence $(1) \iff (2)$ is a special case of Theorem~\ref{thm:kbyn}, which we will prove later in this section.\\

We emphasize the relation, implied by our theorem, between arrangements of second largest minors and the structure of $\Gamma_{(2,n)}$.  Let $\mathcal{J}\subset{[n]\choose 2}$ be a maximal thrackle, and let
\begin{center}
$T= \{A \in Gr^{+}(2,n)\textrm{ }|\textrm{ }\textrm{the set of largest minors of } A \textrm{ is } \mathcal{J} \}$.
\end{center}
Let $W \in {[n]\choose 2}$. Theorem~\ref{thm:2byn}(2) implies that there exists  $A \in T$ for which $W$ is the second largest minor if and only if there exists a vertex $\mathcal{Q}$ in $\Gamma_{(2,n)}$ that is adjacent to $\mathcal{J}$ such that $W \in \mathcal{Q}$.
\begin{example}
Consider the maximal thrackle $\mathcal{J}$ in
Figure~\ref{thrackexample} appearing in the left part on the
top. Using part (4) of Theorem~\ref{thm:2byn}, we identify the
elements in $[n]\choose 2$ that can be second largest minors, and
denote them by red lines (and this is the second graph at the top of
the figure). Then, on the bottom, we describe the thrackle which resulted by adding the red line and removing one of the edges of $\mathcal{J}$. Those three cases correspond to the three edges that are connected to $\mathcal{J}$ in $\Gamma_{(2,5)}$.
\end{example}
\begin{figure}[h]
\includegraphics[height=1.3in]{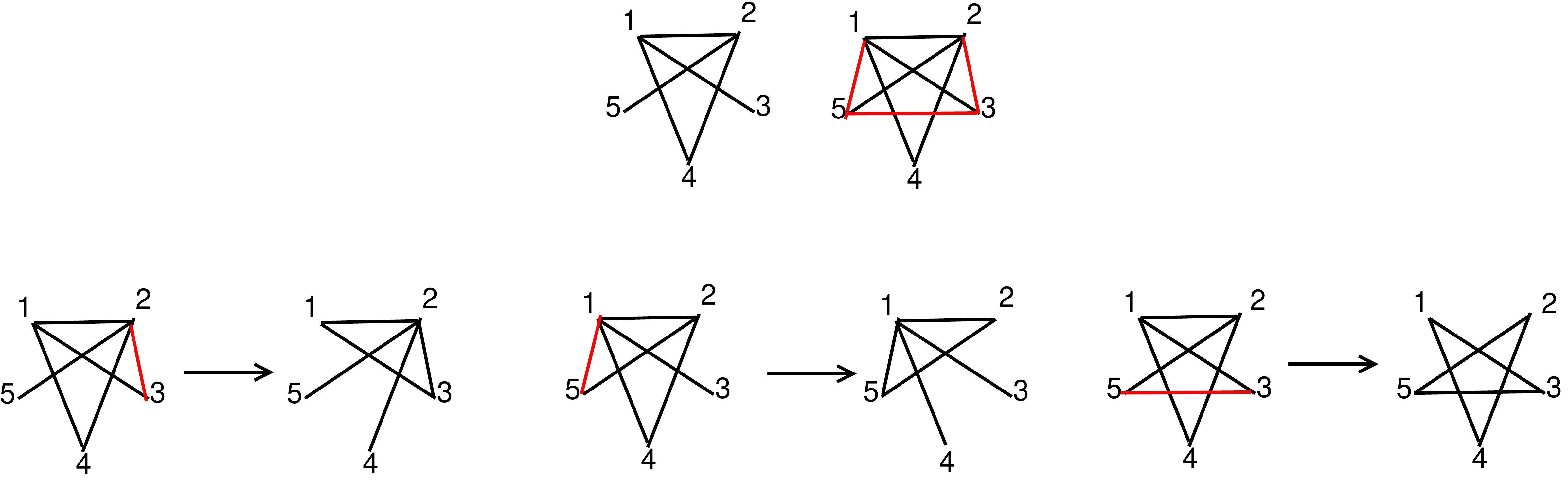}
\caption{}
\label{thrackexample}
\end{figure}
\subsection{The general case}
In the previous subsection, we considered the space\\ $Gr^+(2,n)$ and discussed arrangements of second largest minors when $\mathcal{J}$ was maximal. In this subsection, we consider the space $Gr^+(k,n)$ and discuss arrangements of second largest minors, with no assumption on $\mathcal{J}$. Theorem~\ref{thm:kbynone} summarizes our results. The special case in which $\mathcal{J}$ is maximal will be discussed in Theorem~\ref{thm:kbyn}.
\begin{theorem}
Let $W \in {[n]\choose k}$ and let $\mathcal{J}\subset{[n]\choose k}$ be some arrangement of largest minors such that $W \notin \mathcal{J}$. Denote $|\mathcal{J}|=c$. If $W$ is a $(2,\mathcal{J})-$largest minor, then one of 1,2 holds, or equivalently, one of 3,4 holds:
\begin{enumerate}
  \item The collection $\{W\} \cup \mathcal{J}$ is sorted
  \item There exists $J \in \mathcal{J}$ such that $W$ and $J$ are not sorted, and $(\mathcal{J} \setminus J) \cup \{W\}$ is a sorted collection.
  \item $\nabla_{\{W\} \cup \mathcal{J}}$ is a $c$-dimensional simplex in Sturmfels triangulation of the hypersimplex $\Delta_{k,n}$.
  \item There exists a $c-1$-dimensional simplex $\nabla_\Y$ in Sturmfels triangulation such that $\epsilon_W$ is a vertex in $\nabla_\Y$, and the simplices $\nabla_\Y$,$\nabla_\J$ share a common facet.
\end{enumerate}
\label{thm:kbynone}
\end{theorem}
Before presenting the proof, we will prove the following key lemma:
\begin{lemma}
Let $W,U,V \in {[n]\choose k}$ be three different $k$-tuples, such that the following three conditions hold:
\begin{enumerate}
  \item $U$ and $V$ are sorted.
  \item $W$ and $V$ are not sorted.
  \item $W$ and $U$ are not sorted.
\end{enumerate}
Then the set $T=\{U,V,Sort_1(W,V),Sort_2(W,V),Sort_1(W,U),Sort_2(W,U)\}$ is not sorted.
\label{lem:kbynone}
\end{lemma}
\begin{proof}
Given $I \in {[n]\choose k}$ and $1 \leq i \leq j \leq n$, define $I_{ij}=\sum_{t=i}^{j} (\epsilon_{I})_t$. For example, if $I=\{1,3,5,7,8\}\in {[9]\choose 5}$ then $\epsilon_I=101010110$ and $I_{37}=3$. By definition, for a pair of $k$-tuples $I,J \in {[n]\choose k}$ we have
$$\{Sort_1(I,J)_{ij},Sort_2(I,J)_{ij}\}=\Bigg\{\lfloor {\frac{I_{ij}+J_{ij}}{2}}\rfloor, \lceil {\frac{I_{ij}+J_{ij}}{2}}\rceil\Bigg\}$$
(not necessarily respectively). In particular, if $I$ and $J$ are sorted then $Sort_1(I,J)_{ij}$ and $Sort_2(I,J)_{ij}$ differ by at most 1. In order to prove the lemma, assume for contradiction that $T$ is sorted, and let $\alpha_{ij}=U_{ij}-W_{ij}$ ,$\beta_{ij}=V_{ij}-W_{ij}$ for all $1 \leq i \leq j \leq n$. Following the discussion above, the parameters $\alpha_{ij}, \beta_{ij}$ satisfy the following properties for all $1 \leq i \leq j \leq n$ (the proof of each one of the properties is given below).
\begin{enumerate}
  \item $|\alpha_{ij}|,|\beta_{ij}| \leq 2$.
  \item If $|\alpha_{ij}| = 2$ or $|\beta_{ij}| = 2$ then $\alpha_{ij}=\beta_{ij}$.
  \item If $|\alpha_{ij}| = 1$ then $\alpha_{ij}=\beta_{ij}$ or $\beta_{ij}=0$.
  \item If $|\beta_{ij}| = 1$ then $\alpha_{ij}=\beta_{ij}$ or $\alpha_{ij}=0$.
\end{enumerate}
\textbf{Property 1}: We have $$\frac{W_{ij}+U_{ij}}{2}=\frac{U_{ij}-W_{ij}+2W_{ij}}{2}=\frac{\alpha_{ij}}{2}+W_{ij}.$$ The assumption that $T$ is sorted implies that the pair $Sort_1(W,U), W$ is sorted, as well as the pair $Sort_2(W,U), W$. Therefore, $W_{ij}$ differs from both $Sort_1(W,U)_{ij}$ and $Sort_2(W,U)_{ij}$ by at most 1, which implies property 1 (the proof for $\beta_{ij}$ is similar).\\
\textbf{Property 2}: Assume without loss of generality that $\alpha_{ij} = 2$ (the other cases can be handled similarly), so $U_{ij}=W_{ij}+2$. In addition, one of $Sort_1(W,V)_{ij},Sort_2(W,V)_{ij}$ equals $\lfloor \frac{\beta_{ij}}{2}+W_{ij}\rfloor$, and since $T$ is sorted we must have $\beta_{ij}=2=\alpha_{ij}$.\\
\textbf{Properties 3 and 4}: Assume without loss of generality that $\alpha_{ij} = 1$. Combining properties 1 and 2, it is enough to show that $\beta_{ij} \neq -1$. We have $U_{ij}=W_{ij}+1$ and $V_{ij}=W_{ij}+\beta_{ij}$, and since $T$ is sorted, $\beta_{ij} \neq -1$.\\

$U$ and $W$ are not sorted, and hence there exist $1 \leq i < j \leq n$ such that\\ $|\alpha_{ij}|=|U_{ij}-W_{ij}|>1$. From property 1 we get that $|U_{ij}-W_{ij}|=2$. Recall that $U_{1n}=W_{1n}$, so after appropriate simultaneous rotation of $U,V$ and $W$ (modulo $n$), we can assume that there exists $1 < j < n$ such that $U_{1j}=W_{1j}+2$ (and by property 2 $V_{1j}=W_{1j}+2$ as well, so $U_{1j}=V_{1j}$). From now on, we assume that $U,V$ and $W$ are rotated appropriately, and that $j$ is maximal with respect to this property (that is, there is no $j'>j$ for which $U_{1j'}=W_{1j'}+2$). We can also assume that  $W \cap U \cap V = \emptyset$ (otherwise we could remove the common elements and prove the lemma for the resulting tuples, which implies the claim for the original $k$-tuples as well). In addition, we can also assume that $U \cup V \cup W = [n]$ (since if some $i \in [n]$ appears in neither of them then we could redefine $U,V$, and $W$ to be in ${[n-1]\choose k}$ and ignore this $i$). We divide the proof into a series of claims. Our purpose is to show that the assumption that $T$ is sorted implies $U=V$, which leads to a contradiction.\\

\textbf{Claim 1:}$\{1,j\} \subset U$, $\{1,j\} \subset V$, $\{1,j\} \cap W = \emptyset$.\\

Proof: Assume for contradiction that $1 \in W$. Since $W \cap U \cap V = \emptyset$ then WLOG $1 \notin U$. Therefore $U_{2j}=W_{2j}+3$, contradicting property 1. Similarly $j \notin W$, hence $\{1,j\} \cap W = \emptyset$. Since $U \cup V \cup W = [n]$, we can assume WLOG that $1 \in V$. If $1 \notin U$ then $U_{2j}=W_{2j}+2$, while $V_{2j}=W_{2j}+1$, contradicting property 2. Therefore $1 \in V$, and similarly $j \in V \cap U$, so Claim 1 is proven.\\

\textbf{Claim 2:} For $j < t \leq n$, $(\epsilon_U)_t=(\epsilon_V)_t$.\\

Proof: We prove it by induction on $t$.  For $t=j+1$ property 1 implies that $t \in W$ (otherwise either $V_{1,j+1}=3+W_{1,j+1}$ or $U_{1,j+1}=3+W_{1,j+1}$, a contradiction). If $t \in V$ then since $W \cap U \cap V =\emptyset$ we have $t \notin U$. Therefore,
\begin{center}
$V_{1,j+1}=U_{1,j+1}+1$ and $V_{1,j+1}=W_{1,j+1}+2$,
\end{center}
a contradiction to property 2. Hence $t \notin V$, and similarly $t\notin U$, and the base case of the induction is proven. Assume that the claim holds for all $j+1 \leq t<c$, and let $t=c$. By the inductive assumption $V_{j+1,c-1}=U_{j+1,c-1}$, so applying properties 1 and 2 on $\alpha_{1,c-1}, \beta_{1,c-1}, \alpha_{j+1,c-1}, \beta_{j+1,c-1}$ leads us to the following three cases:
$$\textrm{a) } W_{j+1,c-1}=V_{j+1,c-1}=U_{j+1,c-1},$$
$$\textrm{b) } W_{j+1,c-1}=V_{j+1,c-1}+1=U_{j+1,c-1}+1,$$
$$\textrm{c) } W_{j+1,c-1}=V_{j+1,c-1}+2=U_{j+1,c-1}+2.$$
Case a) contradicts the maximality of $j$, so consider case b). If $c \notin W$ then applying property 2 on $\alpha_{1c}, \beta_{1c}$ implies $c \in U \cap V$, and hence $U_{1c}=W_{1c}+2$- contradicting the maximality of $j$. If $c \in W$ then applying property 2 on $\alpha_{j+1,c}, \beta_{j+1,c}$ implies $c \notin U, c \notin V$, so Claim 2 holds. Let us now consider case c). From property 1 for $\alpha_{j+1,c}, \beta_{j+1,c}$ we must have $c \notin W$, and hence from property 2, $c \in U \cap V$. Thus the claim is proven.\\

\textbf{Claim 3:} For $1 < t \leq j$, $(\epsilon_U)_t=(\epsilon_V)_t$.\\

Proof: We prove it by induction on $t$. The case $t=j$ follows from Claim 1, so assume that the claim is proven for $c < t \leq j$, and let $t=c$. By the inductive hypothesis $U_{c+1,j}=V_{c+1,j}$, and from the proof of Claim 2 it follows that $$\{j+1,j+2\} \subset W, \{j+1,j+2\} \notin U, \{j+1,j+2\} \notin V.$$
Applying properties 1 and 2 on $\alpha_{c+1,j},\beta_{c+1,j},\alpha_{c+1,j+2},\beta_{c+1,j+2}$ leads us to the following options:
$$\textrm{a) } U_{c+1,j}=V_{c+1,j}=W_{c+1,j},$$
$$\textrm{b) } U_{c+1,j}=V_{c+1,j}=W_{c+1,j}+1,$$
$$\textrm{c) } U_{c+1,j}=V_{c+1,j}=W_{c+1,j}+2.$$
First consider case a). In this case we must have $c \in  U \cap V$
and $c \notin W$ (otherwise we get a contradiction when applying properties 1 and 2 on $\alpha_{c,j+2},\beta_{c,j+2}$ ). In case b), if $c \in W$ then the properties of $\alpha_{c,j+2},\beta_{c,j+2}$ imply $c \notin U, c \notin V$. On the other hand, if $c \notin W$ then the properties of $\alpha_{c,j},\beta_{c,j}$ imply $c \in V \cap U$. Finally, in case c) we must have $c \in W, c \notin V, c \notin U$ (by considering $\alpha_{c,j},\beta_{c,j}$), so Claim 3 holds.\\

Combining claims 1,2 and 3 leads us to the conclusion that $\epsilon_U=\epsilon_V$, so $U=V$, a contradiction. Therefore $T$ is not sorted.\\
\end{proof}
We are now ready to present the proof of theorem~\ref{thm:kbynone}.
\begin{proof}
Conditions 1 and 3 are equivalent, as well as conditions 2 and 4. If $W$ is sorted with all the elements in $\mathcal{J}$ then condition 1 holds and we are done. Otherwise, we need to show that $W$ is not sorted with exactly one element in $\mathcal{J}$ (which implies that condition 2 holds). Assume for contradiction that $W$ is not sorted with $U$ and $V$ for some $U,V \in \mathcal{J}$. Since $W$ is a $(2,\mathcal{J})-$largest minor, then there exists $A \in Gr^{+}(k,n)$ such that 1 is the largest value of a Pl\"ucker coordinate in $A$, and $\Delta_I(A)=1$ if and only if $I \in \mathcal{J}$. Moreover, if for some $I\in {[n]\choose k}$ we have $\Delta_W(A)<\Delta_I(A)$, then $I \in \mathcal{J}$ and $\Delta_I(A)=1$. Consider the set $$T=\{U,V,Sort_1(W,V),Sort_2(W,V),Sort_1(W,U),Sort_2(W,U)\}.$$ By Lemma~\ref{lem:kbynone} $T$ is not sorted, and hence $T$ is not contained in $\mathcal{J}$ (since $\mathcal{J}$ is sorted by Theorem~\ref{thm:sorted}). Without loss of generality, $Sort_1(W,U) \notin \mathcal{J}$, so\\ $\Delta_{Sort_1(W,U)} < 1$. Therefore, by Corollary~\ref{cor:Skandera} $$\Delta_W(A)\Delta_U(A) < \Delta_{Sort_1(W,U)}(A)\Delta_{Sort_2(W,U)}(A).$$ Recall that $\Delta_U(A)=1, \Delta_{Sort_2(W,U)}(A) \leq 1$, so
\begin{center}
$\Delta_W(A) < \Delta_{Sort_1(W,U)}(A)\Delta_{Sort_2(W,U)}(A) \leq \Delta_{Sort_1(W,U)}(A)$.
\end{center}
In conclusion
\begin{center}
$\Delta_W(A) < \Delta_{Sort_1(W,U)}(A) < 1$,
\end{center}
contradicting the fact that the value of $\Delta_W(A)$ is second largest among the Pl\"ucker coordinates in $A$.
Therefore $W$ is not sorted with exactly one element in $\mathcal{J}$, and condition 2 holds.
\end{proof}
The theorem above gives a necessary condition on second largest minors. If $\mathcal{J}$ from Theorem~\ref{thm:kbynone} is maximal, we obtain sufficient conditions as well. The following generalizes Theorem~\ref{thm:2byn}.
\begin{theorem}
Let $W \in {[n]\choose k}$ and let $\mathcal{J}\subset{[n]\choose k}$ be some maximal arrangement of largest minors such that $W \notin \mathcal{J}$. The following two statements are equivalent.
\begin{enumerate}
  \item $W$ is a $(2,\mathcal{J})-$largest minor.
  \item There exist a vertex $\mathcal{Q}$ in $\Gamma_{(k,n)}$ that is adjacent to $\mathcal{J}$, such that $W \in \mathcal{Q}$.
\end{enumerate}
In particular, the minors that can be second largest are in bijection with the edges of $\Gamma_{(k,n)}$ that connected to vertex $\mathcal{J}$, and the number of such minors is at most $n$.
\label{thm:kbyn}
\end{theorem}
In order to prove this theorem, we need another result from~\cite{main} which deals with the action of the positive torus on the positive Grassmannian.

\begin{definition}
The {\it ``positive torus''\/} $\R_{>0}^n$ acts on the positive Grassmannian $Gr^+(k,n)$ by rescaling the coordinates in $\R^n$. In terms of $k\times n$ matrices this action is given by rescaling the columns of the matrix.
\end{definition}

\begin{theorem}\cite{main}
{\rm (1)}
For any point $A$ in $Gr^+(k,n)$ and any maximal sorted subset $S\subset {[n]\choose k}$, there is a unique
point $A'$ of $Gr^+(k,n)$ obtained from $A$ by the torus action (that is, by rescaling the columns of the $k\times n$
matrix $A$) such that\\
\medskip
\noindent
{\rm (1)}
The Pl\"ucker coordinates $\Delta_I(A')$, for all $I\in S$, are equal to each other.\\
\medskip
\noindent
{\rm (2)}
All other Pl\"ucker coordinates $\Delta_{J}(A')$, $J\not\in S$,
are strictly less than the $\Delta_I(A')$, for $I\in S$.
\label{thm:torus_action}
\end{theorem}
We now present the proof of Theorem~\ref{thm:kbyn}
\begin{proof}
Theorem~\ref{thm:kbynone} implies $(1) \Rightarrow (2)$. In order to show that $(2) \Rightarrow (1)$, we should construct an element $A \in Gr^{+}(k,n)$ for which the following 3 requirements hold:
\begin{enumerate}
  \item $\Delta_I(A) \leq 1$ for all $I \in {[n]\choose k}$.
  \item $\Delta_I(A) = 1$ iff $I \in \mathcal{J}$.
  \item $\Delta_W(A) \geq \Delta_I(A)$ for all $I \notin \mathcal{J}$.
\end{enumerate}
By theorems~\ref{thm:sorted} and ~\ref{thm:adjacency} there exists $I_t = \{i_1,i_2,\ldots,i_k\} \in \mathcal{J}$ such that
$$W = \{i_1,\ldots,i'_a,\ldots,i'_b,\ldots,i_k\} \textrm{ for some } a < b \in [n]$$
and $i'_a \neq i'_b$, $i_a-i'_a=i'_b-i_b=\pm 1 (mod \textrm{ } n)$, and also both
$k$-subsets $$I_{c}=\{i_1,\ldots,i'_a,\ldots,i_b,\ldots,i_k\}, I_{d}=\{i_1,\ldots,i_a,\ldots,i'_b,\ldots,i_k\}$$ are in $\J$. Let $B \in Gr^{+}(k,n)$ be some element, and let $B'$ be the element that is obtained from $B$ after multiplying the $i$-th column of $B$ by the variable $\alpha_i$ (for all $1 \leq i \leq n$). Then $$\Delta_W(B')=(\Pi_{j=1}^{a-1} \alpha_{i_j})\alpha_{i'_a}(\Pi_{j=a+1}^{b-1} \alpha_{i_j})\alpha_{i'_b}(\Pi_{j=b+1}^{k} \alpha_{i_j})\Delta_W(B)=$$ $$=\frac{\Delta_{I_c}(B')\Delta_{I_d}(B')}{\Delta_{I_t}(B')}\frac{\Delta_{I_t}(B)\Delta_{W}(B)}{\Delta_{I_c}(B)\Delta_{I_d}(B)}.$$
By Theorem~\ref{thm:torus_action} we can choose the scalars $\{\alpha_i\}_{i=1}^n$ in such a way that $\Delta_I(B')=1$ for all $I \in \mathcal{J}$. Therefore, for such a set of scalars,
$$\Delta_W(B')=\frac{\Delta_{I_t}(B)\Delta_{W}(B)}{\Delta_{I_c}(B)\Delta_{I_d}(B)}.$$
Since $i_a-i'_a=i'_b-i_b=\pm 1 (mod \textrm{ } n)$, assume WLOG that $i'_a=i_a-1, i'_b=i_b+1$. Then using three term Pl\"ucker relations we get
\begin{equation}\label{eqproof}
\Delta_W(B')=1-\frac{\Delta_{i_1,i_2,\ldots,i_{a-1},i_{a}-1,i_a,\ldots,i_k}(B)\Delta_{i_1,i_2,\ldots,i_{b-1},i_{b},i_b+1,\ldots,i_k}(B)}{\Delta_{I_c}(B)\Delta_{I_d}(B)}.
\end{equation}
By Theorems~\ref{thm:adjacency} and~\ref{thm:kbynone}, the second largest minor in $B'$ must be obtained from the circuit of $\mathcal{J}$ by a detour. Hence, in order to show that $\Delta_W$ can be second largest,
it is enough to show that we can choose the initial matrix $B$ in such a way that $\Delta_W(B')$ is the biggest among all the minors obtained by a detour. For this purpose we need to maximize the
RHS of (\ref{eqproof}). Let us choose some $C \in Gr^{+}(k,n)$, and denote by $\{C_i\}_{i=1}^n$ its columns. Let $C' \in Gr^{+}(k,n)$ be an element for which
\begin{equation}\label{phi}
C'_j=\left\{
                     \begin{array}{ll}
                       C_j, & \hbox{if $j \notin \{i_a,i_b+1\}$;} \\
                       C_{i_a-1}+\epsilon C_{i_a}, & \hbox{if $j=i_a$;} \\
                       C_{i_b}+\epsilon C_{i_b+1}, & \hbox{if $j=i_b+1$}
                     \end{array}
                   \right.
\end{equation}
for small $\epsilon$. By setting $B=C'$ and using (\ref{eqproof}) to evaluate $\Delta_W(B')$ (and any other minor that is obtained from a detour) one can verify that $\Delta_W(B')=1-O(\epsilon^2)$ while the other minors (that obtained from a detour) are of the order $1-O(\epsilon)$ or $1-O(1)$. Therefore by choosing $\epsilon$ small enough, we obtained an element
$$A=B' \in Gr^{+}(k,n)$$ that satisfies the requirements stated in the beginning of the proof.

\end{proof}

\section{Arrangements of $t$-th largest minors}
Theorem~\ref{thm:kbyn} states that when $\mathcal{J}$ is a maximal
sorted set, the second largest minor must appear in one of the
neighbors of $\mathcal{J}$ in $\Gamma_{(k,n)}$. A natural question is
what can be said regarding $t$-th largest minors for general $t$, and
this is the topic of this section. In the first part, we will define
the notion of cubical distance on $\Gamma_{(k,n)}$, and state our
conjecture regarding $(t,\mathcal{J})$-largest minors. In the second
part, we will prove special cases of this conjecture, and also discuss
the structure of a natural partial order on minors. In the third
part, we discuss additional properties of arrangements of $t$-th
largest minors, and among other things show that they must lie within a certain ball in $\R^n$

\subsection{Cubical distance in $\Gamma_{(k,n)}$}
Consider the blue edges in Figure~\ref{graphGamma2n}, and note that they form a square, while the red edges form a 3-dimensional cube. We say that two vertices $\mathcal{J}_1,\mathcal{J}_2$ in $\Gamma_{(k,n)}$ are of \emph{cubical distance} 1 if both of them lie on a certain cube (of any dimension). For example, vertices $a$ and $b$ from Figure~\ref{graphGamma2n} are of cubical distance 1 since both of them lie on a 1-dimensional cube (which is just an edge). similarly, $a$ and $c$ are of cubical distance 1 (both of them lie on a square), as well as $c$ and $d$ (both of them lie on a 3-dimensional cube).
\begin{definition}
Let $\mathcal{J}_1,\mathcal{J}_2\subset{[n]\choose k}$ be maximal sorted collections, and let $W \in {[n]\choose k}$. We say that $\mathcal{J}_1,\mathcal{J}_2$ are of cubical distance $D$, and denote it by $cube_d(\mathcal{J}_1,\mathcal{J}_2)=D$, if one can arrive from $\mathcal{J}_1$ to $\mathcal{J}_2$ by moving along $D$ cubes in $\Gamma_{(k,n)}$, and $D$ is minimal with respect to this property. We say that $W$ is of cubical distance $D$ from $\mathcal{J}_1$, and denote it by $cube_d(\mathcal{J}_1,W)=D$, if for any vertex $\mathcal{J}_2$ in $\Gamma_{(k,n)}$ that contains $W$, $cube_d(\mathcal{J}_1,\mathcal{J}_2) \geq D$, and for at least one such $\mathcal{J}_2$ this inequality becomes equality.
\end{definition}
For example, using the notations of Figure~\ref{graphGamma2n}, $cube_d(a,d)=2$, $cube_d(b,d)=2$, $cube_d(a,e)=3$. We also have $cube_d(a,\{1,4\})=1$ since $\{1,4\} \in f$. Similarly, $cube_d(a,\{2,4\})=2$ since $\{2,4\} \notin b,f,c$, and $\{2,4\} \in d$. It can also be shown that $cube_d(a,\{2,3\})=3$.
\begin{definition}
Let $\mathcal{J}\subset{[n]\choose k}$ be an arrangement of largest minors, and let\\ $W \in {[n]\choose k}$. We say that $W$ is $(\geq t,\mathcal{J})-$largest minor if for any arrangement of minors $\U=(\U_0,\U_1, \dots,\U_l)$ such that $\U_l=\mathcal{J}$ the following holds:\\
$W \notin \U_l,\U_{l-1},\ldots, \U_{l-t+2}$.
\end{definition}
For example, let $\mathcal{J}$ be the maximal sorted set that corresponds to vertex $a$ in Figure~\ref{graphGamma2n}, and let $A \in Gr^+(2,6)$ in which the collection of maximal minors is $\mathcal{J}$. Using Skandera's inequalities (Corollary~\ref{cor:Skandera}), it is possible to show that for such $A$, $\Delta_{16}>\Delta_{14}>\Delta_{24}>\Delta_{23}$. Therefore, $\{2,3\}$ is $(\geq 4,\U_{l})-$largest minor, since $\{2,3\} \notin \U_l,\U_{l-1}, \U_{l-2}$.
\begin{conjecture}\label{conjecture1}
Let $W \in {[n]\choose k}$ and let $\mathcal{J}\subset{[n]\choose k}$ be some maximal arrangement of largest minors such that $W \notin \mathcal{J}$. If $cube_d(W,\mathcal{J})=t$, then $W$ is $(\geq t+1,\mathcal{J})-$largest minor.
\end{conjecture}
Note that the examples we gave earlier are special cases of this conjecture. For example, $cube_d(a,\{2,3\})=3$, and indeed $\{2,3\}$ is $(\geq 4,\U_{l})-$largest minor.
In many cases, we can prove this conjecture. Our main results in this section are Theorems~\ref{conjecturespecial} and~\ref{finaltheorem}, both of them validate the conjecture for wide class of cases.
\begin{theorem}\label{conjecturespecial}
Conjecture~\ref{conjecture1} holds for $t=2,3$ (and any $n$ ,$k$), and also for $k=2$ (and any $n$, $t$).
\end{theorem}
\begin{theorem}\label{finaltheorem}
If $W$ is sorted with at least one element in $\mathcal{J}$, then Conjecture~\ref{conjecture1} holds.
\end{theorem}
At first glance, it may seem like Theorem~\ref{conjecturespecial} contradicts Theorem~\ref{thm:kbyn} since a vertex $\mathcal{J}_2$ of cubical distance 1 from $\mathcal{J}$ doesn't have to be connected to $\mathcal{J}$. However, using Theorem~\ref{thm:adjacency}, it can be easily shown that if $W \in \mathcal{J}_2$ then $W$ also appears in one of the neighbors of $\mathcal{J}$.
\subsection{Partially ordered set of minors}
In this part, we show that arrangements of largest minors induce a
structure of a partially ordered set on the entire collection of minors. The investigation of this poset leads us to the proof of Theorem~\ref{finaltheorem}. We conclude this part with the proof of Theorem~\ref{conjecturespecial}.
\begin{example}\label{examplepos}
Let $k=2,n=6$, and let $A \in Gr^+(2,6)$ be an element for which the minors that appear in Figure~\ref{lastexample} on the left are maximal. Thus, without loss of generality, we can assume that
\begin{center}
$\Delta_{12}=\Delta_{13}=\Delta_{14}=\Delta_{15}=\Delta_{25}=\Delta_{26}=1$.
\end{center}
\begin{figure}[h]
\centering
\includegraphics[height=1.2in]{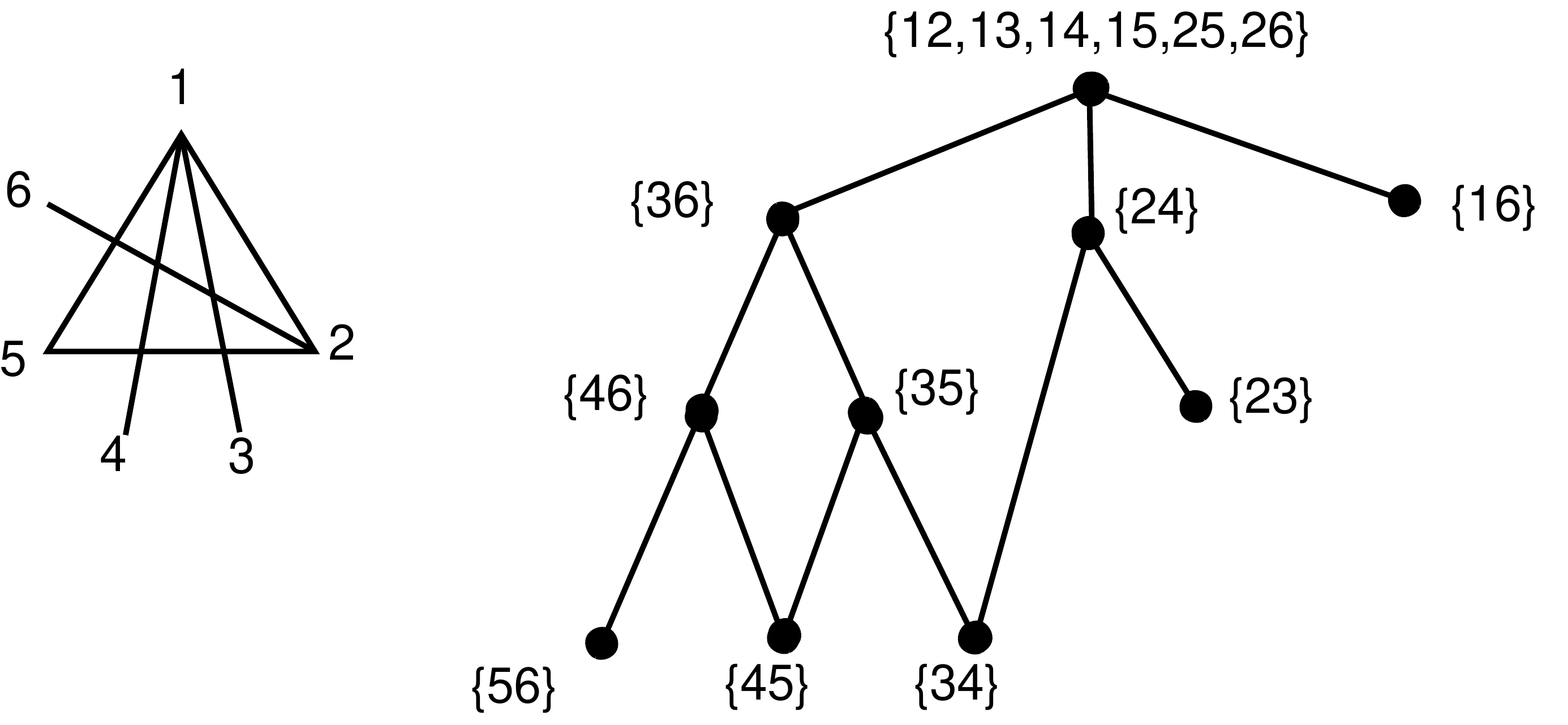}
\caption{A maximal thrackle and the corresponding poset of minors}
\label{lastexample}
\end{figure}
By Theorem~\ref{thm:sorted}, all the other minors are strictly smaller than 1. However, there is much more information
that we can obtain on the order of the minors. For example, using 3-term Pl\"ucker relations, we get $\Delta_{46}\Delta_{13}<\Delta_{14}\Delta_{36}$, and hence $\Delta_{46}<\Delta_{36}$. Once the set of largest minors is fixed, it induces a partial order on the entire collection of minors. Figure~\ref{lastexample} depicts the Hasse diagram that corresponds to the example above (and the relation $\Delta_{46}<\Delta_{36}$ is one of the covering relations in this diagram).


\end{example}
In order to discuss these partially ordered sets more systematically, and to prove Theorem~\ref{finaltheorem}, we will use the circuit triangulation of the hypersimplex, introduced in section~\ref{sec:Triangulations}. The structure of $G_{k,n}$ is quite complicated in general. Yet, we found an algorithm that recognizes certain planar subgraphs of $G_{k,n}$ which induce the partial order.
\begin{definition}\label{maindefinition}
An \emph{oriented Young graph} is the graph that is obtained from a Young diagram after rotating it in 180 degrees and orienting each horizontal edge from left to right and each vertical edge from bottom to top.
We call the vertex that is in the lower right corner the \emph{origin vertex}, and denote the upper right
(lower left) vertex by $v_1$ ($v_0$). There are two paths that start at $v_0$, continue along the border and end at
$v_1$. The path that passes through the origin vertex is called \emph{inner path}, and the second path is called
\emph{outer path}. From now on, we denote the set of the vertices appear in the outer path by $V$.
See Figure~\ref{gridgraph} for an example.
\begin{figure}[h]
\centering
\includegraphics[height=1.25in]{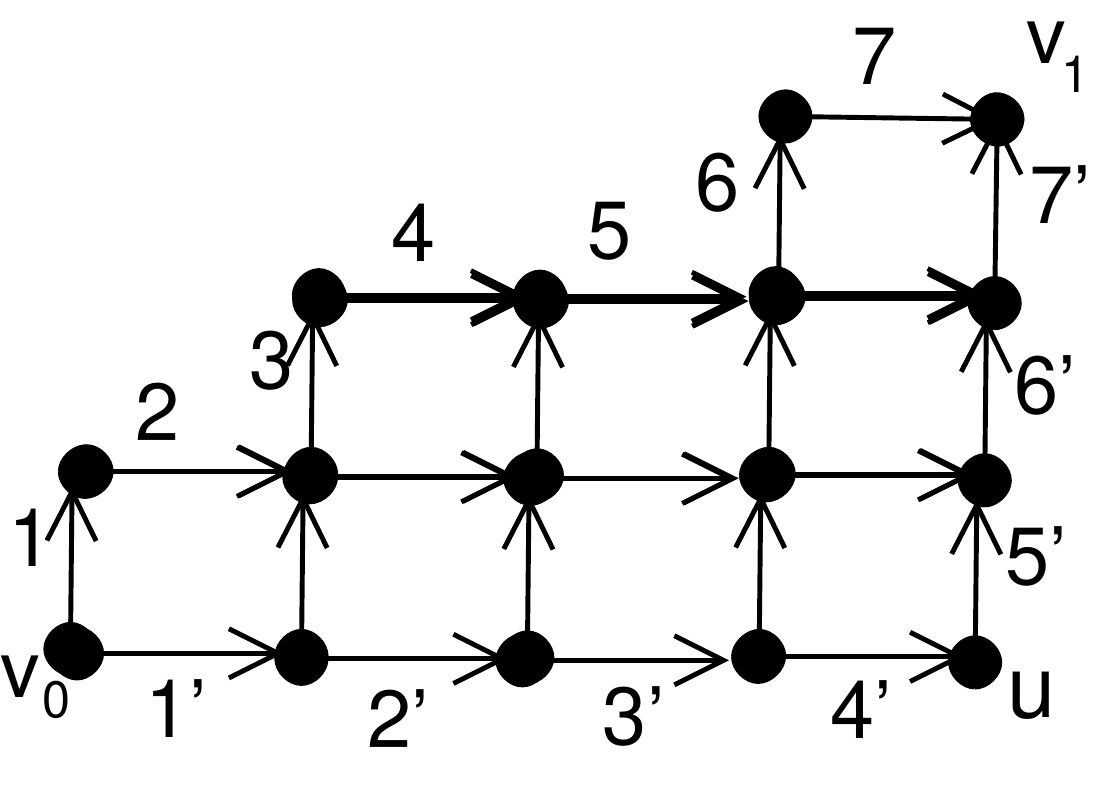}%
\caption{An oriented Young graph. Its inner boundary path is formed by the edges labeled from 1' through 7'.
Its outer boundary path is formed by the edges labeled from 1 through 7, and all the vertices that appear along the latter path form the collection V.}
\label{gridgraph}
\end{figure}
\end{definition}
\begin{lemma}\label{zigzaglemmaone}
Let $H$ be an oriented Young subgraph of $G_{k,n}$, and let $T \in Gr^+(k,n)$ for which all
the minors indexed by $V$ are equal and have largest value. Then for any vertex $D$ of $H$ such that
$D \notin V$, we have $$\Delta_D(T)<\Delta_C(T) \textrm{ and }\Delta_D(T)<\Delta_A(T),$$ where $C$ is the vertex right above
$D$ and $A$ is the vertex to the left of $D$ in $H$ (see Figure~\ref{square}).
\end{lemma}
\begin{figure}[h]
\centering
\includegraphics[height=0.7in]{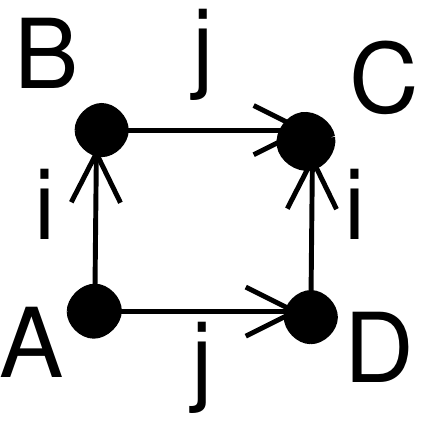}%
\caption{}
\label{square}
\end{figure}
Before presenting the proof of the lemma, we would like to present the proof idea of Theorem~\ref{finaltheorem} using the running example depicted in Figure~\ref{cubegraph}.
\begin{figure}[h]
\centering
\includegraphics[height=4.1in]{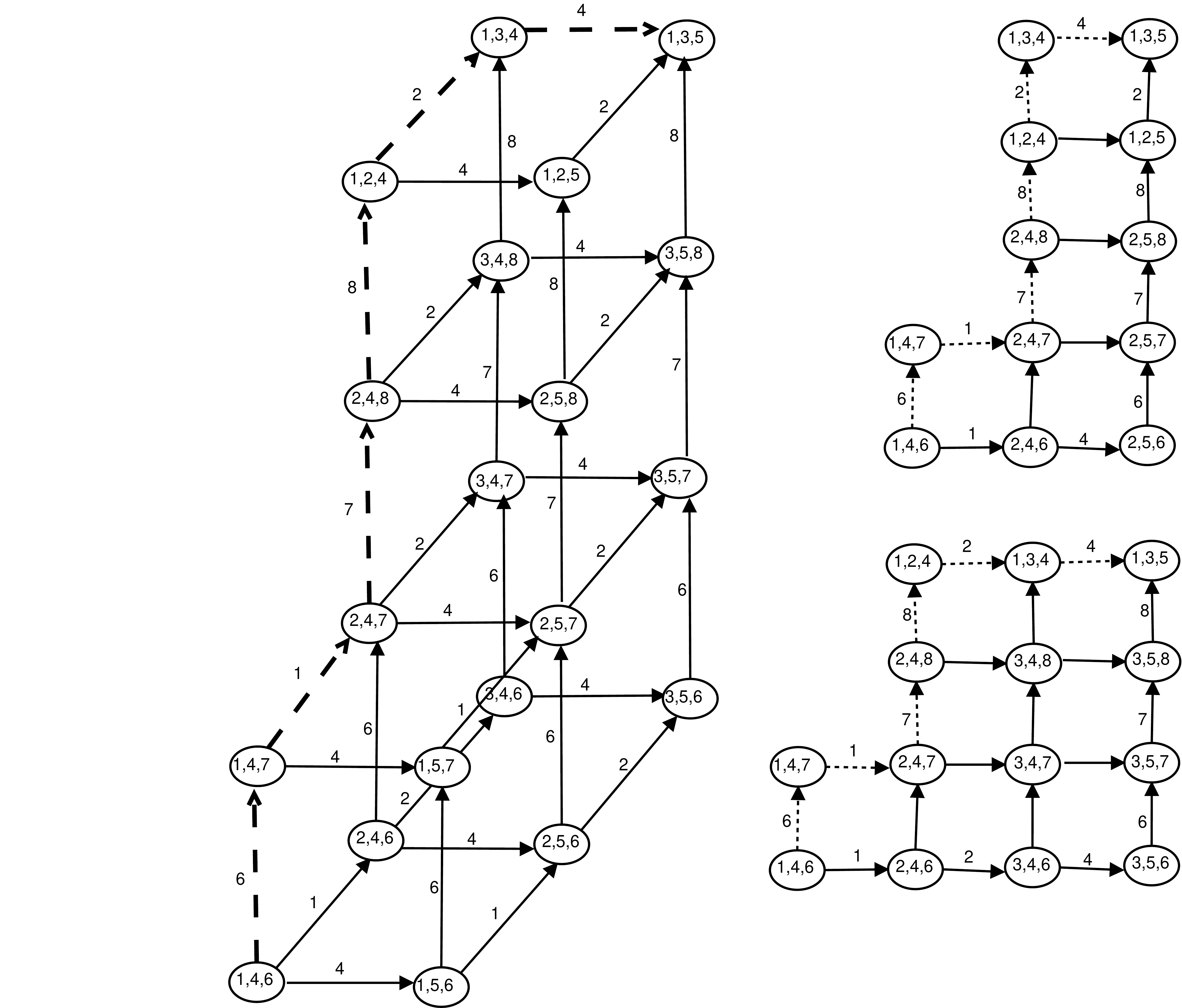}
\caption{The graph on the left is $Q_1$. The graph on the top right is $Q_2$, and the graph on the bottom right is $Q_3$.}
\label{cubegraph}
\end{figure}
The proof will show that under the conditions of the theorem, one can find an oriented
Young subgraph of $G_{k,n}$ such that $W$ is the origin vertex and $V \subset \J$. Then we apply
Lemma~\ref{zigzaglemmaone} and obtain an ordering on the minors. As an example,
suppose that the minors corresponding to the circuit $C_{\omega}$ in Figure~\ref{circuit38}
form an arrangement $\J$ of largest minors, and let $W=(3,5,6)$. One can verify that $cube_d(\mathcal{J},W)\leq 4$. Among the vertices of $C_{\omega}$, $W$ is sorted with $\{1,3,5\}, \{1,4,5\}, \{1,4,6\}$, and not sorted with the rest. So the set of vertices that are not sorted with $W$ form a path in $C_{\omega}$
(and this property also holds in the general case as we will show in Lemma~\ref{lemmpath}). We would like to construct an alternative path in $G_{3,8}$ that starts at $ \{1,4,6\}$, ends at $\{1,3,5\}$, passes through $W$, and contains only vertices that are sorted with $W$. Consider the left graph $Q_1$ that appears in Figure~\ref{cubegraph}.
$Q_1$ is a subgraph of the graph $G_{3,8}$, and the edges that correspond to the circuit $C_{\omega}$ appear
as dotted lines. The part of $\omega$ that corresponds to the dotted lines is $617824$ (we ignore the vertex $\{1,4,5\}$, as it is sorted with $W$). Consider the path that starts at $\{1,4,6\}$ and continues along the edges labeled by $124678$. Note that after 3 steps in this path, we arrive to the vertex $W$. $Q_3$ (see Figure~\ref{cubegraph}) is the oriented Young subgraph of $G_{3,8}$ in which the set $V$ consists of vertices from $C_{\omega}$ and $W$ is the origin vertex. One can check that this is indeed a subgraph of $Q_1$. Applying Lemma~\ref{zigzaglemmaone} we get $$\Delta_{3,5,6}<\Delta_{3,4,6}<\Delta_{3,4,7}<\Delta_{3,4,8}<\Delta_{1,3,4}.$$
This implies that $W$ is $(\geq 4+1,\mathcal{J})-$largest minor, as guaranteed by Theorem~\ref{finaltheorem}.
Similar claim holds in the case $W=(2,5,6)$, with the corresponding oriented Young subgraph $Q_2$.

The proof of Theorem~\ref{finaltheorem} will be based on several lemmas. We start by presenting the proof of Lemma~\ref{zigzaglemmaone}.
\begin{proof}
Without loss of generality we can assume that $\Delta_M(T)=1$ for all $M \in V$, and that 1 is the largest
minor of $T$. Consider the subgraph of $H$ that looks like the graph in Figure~\ref{square}. Then the labelings of its edges
(the labeling that induced from $G_{k,n}$) must look as in the figure.
The proof is by induction on the distance $d$ of the vertex $D$ from
the vertices in $V$, when distance is defined as the sum of vertical and horizontal path from the vertex to the vertices of
$V$. We denote this distance by $d(D,V)$. For example, the distance of vertex $u$ from $V$ in Figure~\ref{gridgraph} is 3+4=7, as the vertical path
has 3 edges and the horizontal path has 4 edges. The base case of the induction is the case $d=2$. In such a case, $A,B,C \in V$.
Moreover, using the labelings in Figure~\ref{square}, $A,B,C,D$ are of the form:
$$A=\{a_1,a_2,\ldots,a_m,i,a_{m+2},\ldots,a_p,j,a_{p+2},\ldots,a_k\},$$
$$B=\{a_1,a_2,\ldots,a_m,i+1,a_{m+2},\ldots,a_p,j,a_{p+2},\ldots,a_k\},$$
$$C=\{a_1,a_2,\ldots,a_m,i+1,a_{m+2},\ldots,a_p,j+1,a_{p+2},\ldots,a_k\},$$
$$D=\{a_1,a_2,\ldots,a_m,i,a_{m+2},\ldots,a_p,j+1,a_{p+2},\ldots,a_k\}.$$
Applying 3-term Pl\"ucker relation we get $\Delta_D(T)\Delta_B(T)<\Delta_A(T)\Delta_C(T)$, and since $A,B,C \in V$ we have
$\Delta_D(T)<1$. This implies $\Delta_D(T)<\Delta_C(T)$ and\\ $\Delta_D(T)<\Delta_A(T)$, so we are done with the base case. Suppose now that the distance
is $d=d(D,V)>2$. Clearly $$d(D,V)<d(A,V), d(D,V)<d(C,V) \textrm{ and } d(D,V)<d(B,V),$$ so we can apply the inductive hypothesis on $A,B$ and $C$, and get\\ $\Delta_C(T)<\Delta_B(T), \Delta_A(T)<\Delta_B(T)$. Hence
$$\Delta_D(T)\Delta_B(T)<\Delta_A(T)\Delta_C(T)<\Delta_B(T)\Delta_C(T),$$
so $\Delta_D(T)<\Delta_C(T)$. Similarly,
$$\Delta_D(T)\Delta_B(T)<\Delta_A(T)\Delta_C(T)<\Delta_A(T)\Delta_B(T)$$
so $\Delta_D(T)<\Delta_A(T)$ and we are done.
\end{proof}

Given an oriented Young graph $H$ and a vertex $w \in H$, we denote the position of $w$ in $H$ by $(i,j)$
where $i$ and $j$ start at 0 and the origin vertex corresponds to $(0,0)$. For example, in Figure~\ref{gridgraph} the position of $v_1$ is $(3,0)$, the position of $v_0$ is $(0,4)$ and
the position of $u$ is $(0,0)$. In the following section, we sometimes refer to a vertex directly by its position.
\begin{definition}\label{zigzagdef}
Let $H$ be an oriented Young subgraph of $G_{k,n}$, and let $u$ be the origin vertex.
The \emph{swapping distance between $u$ and $V$} is $\max\{i+j-1 | (i,j) \in H\}$.
\end{definition}
For example, the swapping distance of $u$ from $V$ in Figure~\ref{gridgraph} is 4, and it obtained by taking the vertex that is
incident to both edges 3 and 4.
\begin{corollary}\label{corzagdef}
Let $H,V,u$ be as in Lemma~\ref{zigzaglemmaone}, and denote by $t$ the swapping distance of $u$ from $V$.
Let $\U_{l}\subset{[n]\choose k}$ be an arrangement of largest minors such that $V \subset \U_{l}$. Then $u$ is $(\geq t+1,\U_{l})-$largest minor.
\end{corollary}
\begin{proof}
We will prove it by induction on the swapping distance $s$. If it equals 1, then the claim follows immediately. If $s>1$ then there are two options:
\begin{enumerate}
  \item At least one of the points $(1,s), (s,1)$ are in $H$.
  \item Both points $(1,s), (s,1)$ are not in $H$.
\end{enumerate}
Consider case 1. Applying Lemma~\ref{zigzaglemmaone} and assuming WLOG that $(s,1) \in H$, we get
$$\Delta_u < \Delta_{(1,0)}< \Delta_{(2,0)}< \ldots< \Delta_{(s,0)}\leq 1.$$
Therefore $u$ is $(\geq s+1,\U_{l})-$largest minor. \\
Let us now consider case 2. Denote by $(i_1,j_1)$ the vertex in $H$ that maximizes $\{i+j-1 | (i,j) \in H\}$.
Since neither $(s,1)$ nor $(1,s)$ are in $H$, we have $i_1 < s, j_1 < s$. Denote the vertex in position $(1,0)$ by $A$,
and consider the induced subgraph $R$ of $H$ in which $A$ is the origin.
Clearly we have $$\max\{i+j-1 | (i,j) \in R\} = i_1+j_1-2=s-1,$$ so by the inductive hypothesis $A$ is $(\geq s,\U_{l})-$largest minor.
By Lemma~\ref{zigzaglemmaone},\\ $\Delta_u < \Delta_A$, and hence $u$ is $(\geq s+1,\U_{l})-$largest minor, and we are done.
\end{proof}
Our next lemma relates the swapping distance with the cubical distance, defined in the beginning of this section.
\begin{lemma}\label{boundcubic}
Let $\J\subset{[n]\choose k}$ be a maximal sorted collection, and suppose that there exists an oriented
Young subgraph $H$ of $G_{k,n}$ such that $V \subset \J$. Let $u$ be the origin vertex in $H$.
Then $cube_d(\J,u)$ is bounded from above by the swapping distance of $u$ from $V$.
\end{lemma}
Before presenting the proof of this lemma, we would like to clarify the relationship between circuit triangulation and cubical distance. Let $C_p$
and $C_q$ be two minimal circuits. By Theorem~\ref{circuitthm}, the vertices of each one of the circuits form a maximal sorted collection.
We denote these collections by $\mathcal{P}$ and $\mathcal{Q}$
respectively. We leave it as an exercise for the reader to check
that the following claim holds (see also Figure~\ref{detourcircuit}):
\begin{claim}\label{claimcubical}
\begin{enumerate}
  \item $cube_d(\mathcal{P},\mathcal{Q})=1$ if and only if $C_q$ is obtained from $C_p$ by making a set of different detours
$\{{I_c}^i,{I_t}^i,{I_d}^i\}_{i=1}^m$ such that for every pair $1 \leq i < j \leq m$, neither ${I_t}^i$ nor ${I_t}^j$ lie in the intersection\\ $\{{I_c}^i,{I_t}^i,{I_d}^i\} \cap \{{I_c}^j,{I_t}^j,{I_d}^j\}$.
  \item $cube_d(\mathcal{P},\mathcal{Q})=t$ if and only if $C_q$ is obtained from $C_p$ by a sequence of $t$ steps, each one of them
of the form described in (1), such that $t$ is minimal with regard to this property.
\end{enumerate}
\end{claim}
\begin{figure}[h]
\centering
\includegraphics[height=1.25in]{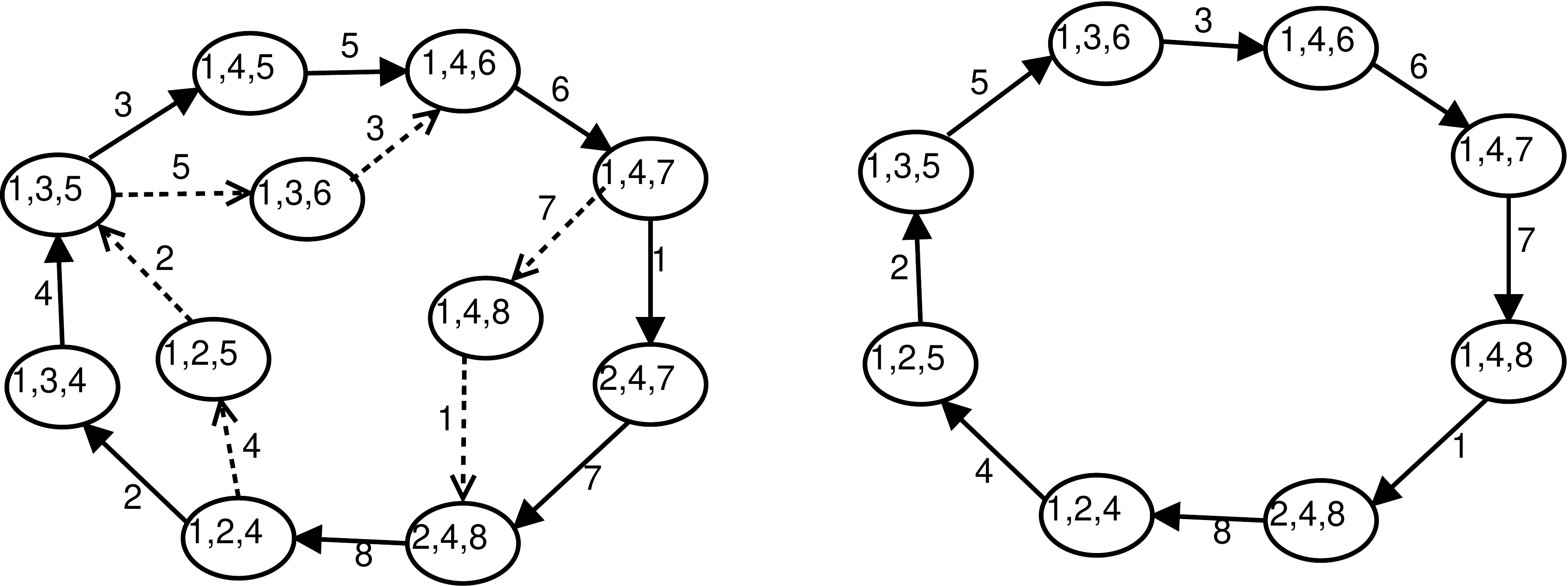}
\caption{The figure on the left is a circuit in $G_{3,8}$ which we have already seen before. There are 3 detours depicted
in dotted lines, and the circuit to the right is the circuit that is obtained by these detours. These two minimal
circuits correspond to a pair of maximal sorted sets of cubical distance 1.}
\label{detourcircuit}
\end{figure}
We will now prove Lemma~\ref{boundcubic}
\begin{proof}
Denote by $s$ the swapping distance of $u$ from $V$. In order to prove this lemma,
we need to show that there exists a maximal sorted collection $\I\subset{[n]\choose k}$
such that $u \in \I$ and such that there exists a sequence of $s$ moves that connects between $\I$ and $\J$ as
described in Claim~\ref{claimcubical} (so each of these moves
corresponds to a certain set of detours). Consider the set of all corner vertices $\{w_i\}_{i=1}^g$ in $V$ ($w$ is a corner vertex if neither the vertex
above $w$ nor the vertex to the left of $w$ is in $V$). Each such corner vertex corresponds to a vertex $B$ in a square as in
Figure~\ref{square}. So we can make a detour that exchanges the arcs $A \rightarrow B$ and $B \rightarrow C$
with the arcs $A \rightarrow D$ and $D \rightarrow C$. Those detours satisfy the requirement in Claim~\ref{claimcubical}
so we can make all of them at the same time. The resulting oriented Young graph has swapping distance $s-1$, so after applying
this process $s$ times we get a maximal sorted set $\I$ that contains $u$ (note that $\I$ and $\J$ are identical on all the vertices
outside $V$), and that completes the proof. See Figure~\ref{sequenceproof} for an example.
\end{proof}
\begin{figure}[h]
\centering
\includegraphics[height=2in]{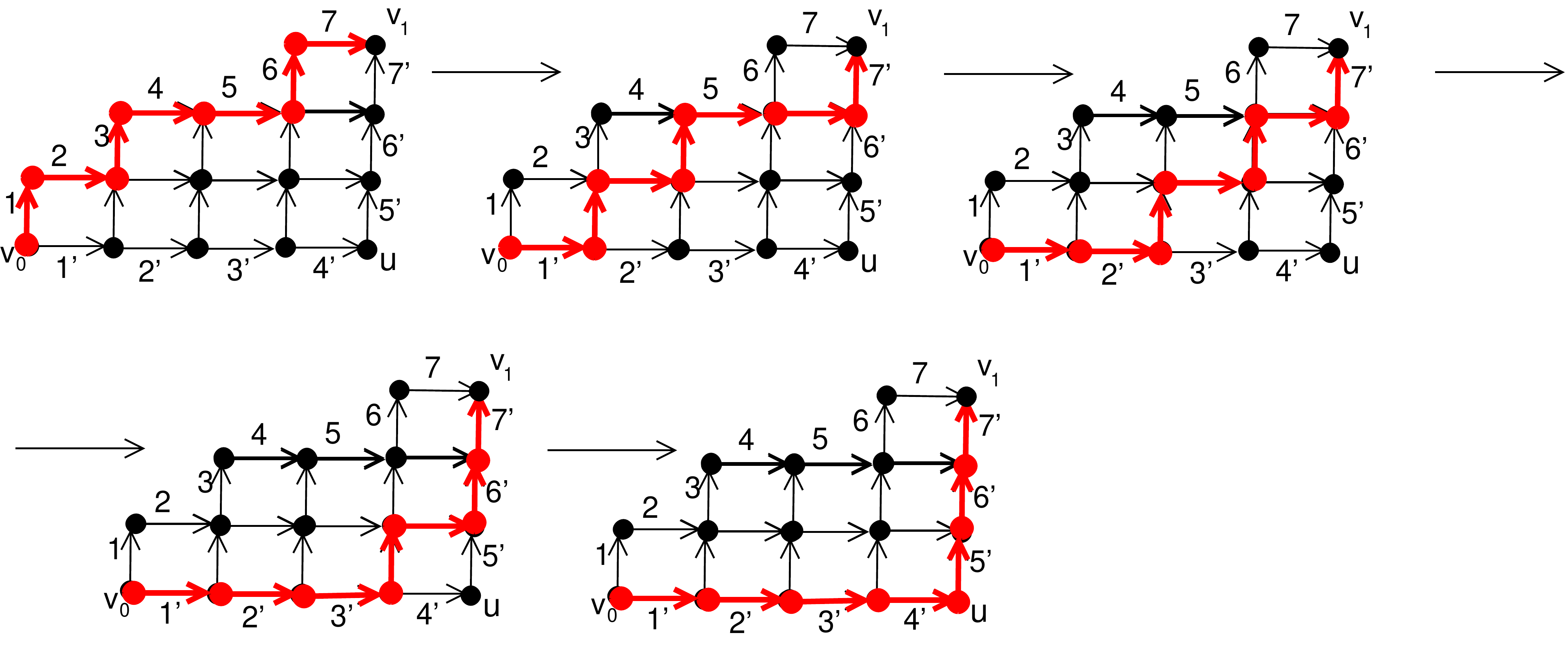}
\caption{The description of the sequence from the proof of Lemma~\ref{boundcubic}}
\label{sequenceproof}
\end{figure}
Our last lemma deals with induced paths in minimal circuits.
\begin{lemma}\label{lemmpath}
Let $C_\omega$ be a minimal circuit in $G_{k,n}$ and let $W \in {[n] \choose k}$, such that $W \notin C_\omega$. Let $B$ be the set of vertices of
$C_\omega$ that are sorted with $W$. Then the induced subgraph of $B$ in $C_\omega$ is a path (which might be empty).
\end{lemma}
As an example, consider the circuit $C_{\omega}$ in Figure~\ref{circuit38}, and let $W=(3,5,6)$. Among the vertices of $C_{\omega}$,
$W$ is sorted with $\{1,3,5\}, \{1,4,5\}, \{1,4,6\}$, which indeed form a path.
\begin{proof}
If $W$ is sorted with exactly 0 or one elements in $C_\omega$ then the statement is clear. Hence assume that $W=\{c_1,\ldots,c_k\}$ is
sorted with two vertices in $C_\omega$: $a=\{a_1,a_2,\ldots,a_k\}$ and $b=\{b_1,b_2,\ldots,b_k\}$.
Since the collection $\{a,b,W\}$ is sorted then by possibly rotating the circle $\{1,2 \ldots, n\}$ and switching the roles
of $a$ and $b$ we can assume WLOG that $$c_1 \leq a_1 \leq b_1 \leq c_2 \leq a_2 \leq b_2 \leq \ldots \leq c_k \leq a_k \leq b_k.$$
We will show that every element in the path from $a$ to $b$ is sorted with $W$. Let
$d=\{d_1,d_2,\ldots,d_k\}$ be an element in this path. Then by the definition of minimal circuit,
$$a_1 \leq d_1 \leq b_1 \leq a_2 \leq d_2 \leq b_2, \ldots,\leq a_k \leq d_k \leq b_k.$$ Therefore,
$$c_1 \leq d_1 \leq c_2 \leq d_2 \leq \ldots \leq c_k \leq d_k$$ and we are done.
\end{proof}

We are now ready to present the proof of Theorem~\ref{finaltheorem}.
\begin{proof}
Suppose that there exists an oriented Young subgraph $H$ of $G_{k,n}$ such that $V \subset \J$ and
$W$ is the origin vertex of $H$. In such a case, if we denote by $s$ the swapping distance of $W$ from $V$, then by
Lemma~\ref{boundcubic} $cube_d(\J,W)\leq s$. On the other hand, Corollary~\ref{corzagdef} implies that
$W$ is $(\geq s+1,\J)-$largest minor. Therefore in particular $W$ is $(\geq cube_d(\J,W)+1,\J)-$largest minor, which is
exactly the statement of Conjecture~\ref{conjecture1}. Hence our purpose in this proof is to construct such $H$.
Denote by $C_{\J}$ the minimal circuit in $G_{k,n}$ that corresponds to the set $\J$, and by
$\omega_{\J}$ the permutation that is associated with $C_{\J}$.
As we mentioned in the proof of Lemma~\ref{boundcubic}, $H$ will actually provide us a minimal circuit $C_H$ in $G_{k,n}$ that contains
$W$ (see also Figure~\ref{cubegraph}, and the discussion regarding this figure following Lemma~\ref{zigzaglemmaone}).
Thus, in order to find such subgraph $H$ it is enough to find the permutation $\omega_H$ which corresponds
to the minimal circuit $C_H$, and to show that the part on which
$C_{\J}$ and $C_H$ differ induces a structure of an oriented Young subgraph.
For example, in Figures~\ref{circuit38} and~\ref{cubegraph}, if $W=(3,5,6)$ then we have $\omega_\J=61782435$, $\omega_H=12467835$, so the part
on which $C_\J$ and $C_H$ differ corresponds to the graph $Q_3$ depicted in Figure~\ref{cubegraph}.\\

We will first give a description of $C_H$, and then prove that it satisfies the requirements. Since $W$ is sorted with
at least one vertex in $C_\J$, then by Lemma~\ref{lemmpath} there exist vertices
$A=\{a_1,\ldots,a_k\}$ and $B=\{b_1,\ldots,b_k\}$ in $C_\J$ such that $W$ is sorted
with all the vertices in the path $B \rightarrow A$ (including the endpoints), and not sorted with all
the vertices in the path $A \rightarrow B$ (excluding the endpoints). We also allow the possibility $A=B$
(in which case $W$ is sorted with exactly one element in $C_\J$. Note that $W$ cannot be sorted with all the elements in
$C_\J$ since $\J$ is maximal). Since $A$ and $B$  are sorted, then by appropriate rotation of the circle
$\{1,2,\ldots,n\}$ we can assume that
\begin{equation}\label{eqab}
a_1 \leq b_1 \leq a_2 \leq b_2 \ldots.... \leq a_k \leq b_k.
\end{equation}
So if $A=\{1,4,6\}$ and $B=\{1,3,5\}$ as in Figure~\ref{circuit38}, then using the order\\
$6<1<2<3<4<5$ we have $6 \leq 1 \leq 1 \leq 3 \leq 4 \leq 5$, and we "redefine" $A$ to be $A=\{6,1,4\}$. In the case $A=B$ we
set $B=\{a_2,a_3,\ldots,a_k,a_1\}$. Let $W=\{d_1,\ldots,d_k\}$ such that
$d_1<d_2<\ldots<d_k$ in the order $$a_1<a_1+1\ldots < n <1 <2 <\ldots a_1-1.$$ We claim that the numbers $\{a_i,b_i,d_i\}_{i=1}^k$ satisfy inequality (\ref{eqacb}) below. We will first show how to use this inequality in order to construct $C_H$, and then in the last paragraph of the proof we will prove this inequality.
\begin{equation}\label{eqacb}
a_1 \leq d_1 \leq b_1 \leq a_2 \leq d_2 \leq b_2 \ldots.... \leq a_k \leq d_k \leq b_k
\end{equation}

Denote the path from $A$ to $B$ in $C_{\J}$ by $Q$, and let $\widehat{\omega}=\omega_1\omega_2\ldots\omega_m$
be the partial permutation that corresponds to $Q$. In particular, $\widehat{\omega}$ is a contiguous part of
$\omega_{\J}=\omega_1\omega_2\ldots\omega_m\omega_{m+1}\ldots\omega_{n}$  (for example, in Figure~\ref{circuit38},
if $W=(3,5,6)$ then\\ $\widehat{\omega}=617824$).Since (\ref{eqab}) holds, then for every $1 \leq i \leq k$, the
"1" in the position $a_i$ in $\epsilon_A$ is shifted along $Q$ to the "1" in the position $b_i$ in $\epsilon_B$.
Define\\ $A_i=\{a_i,a_i+1,\ldots,b_i-1\}$ for all $1 \leq i \leq k$ (where $n+1$ is identified with 1), and note that $A_i=\emptyset$ iff $a_i=b_i$. Then the set
of numbers that appear in $\widehat{\omega}$ is, in fact, $\cup_{i=1}^k A_i$. We would like to use now property
(\ref{eqacb}): For every $1 \leq i \leq k$ define $${D_i}^1=\{a_i,a_i+1,\ldots,d_i-1\},
 {D_i}^2=\{d_i,d_i+1,\ldots,b_i-1\}$$ (this is well defined since $a_i \leq d_i \leq b_i$).
We define $\omega_H$ as follows: Its first part consists of the numbers from $\cup_{i=1}^k {D_i}^1$, placed
according to the order in which they appear in $\widehat{\omega}$. Its second part consists of
the numbers from $\cup_{i=1}^k {D_i}^2$, again placed according to the order in which they appear in $\widehat{\omega}$.
Finally we place $\omega_{m+1}\ldots\omega_{n}$. To make this definition more clear, consider the circuit in
Figure~\ref{circuit38}
and let\\ $W=\{3,5,6\}$. Then $\widehat{\omega}=617824$, $A=\{6,1,4\}, B=\{1,3,5\}$, and we rotate the elements in
$W$ so that $W=\{6,3,5\}$. We have:
$$A_1=\{6,7,8\},  A_2=\{1,2\},  A_3=\{4\},$$  $${D_1}^1=\emptyset,  {D_1}^2=\{6,7,8\},
{D_2}^1=\{1,2\},  {D_2}^2=\emptyset,  {D_3}^1=\{4\},  {D_3}^2=\emptyset.  $$
Therefore, $\cup_{i=1}^k {D_i}^1=\{1,2,4\}, \cup_{i=1}^k {D_i}^2=\{6,7,8\}$. Therefore $\omega_H=12467835$,
and indeed $C_H$ contains $W$ as is shown in the graph $Q_3$ in Figure~\ref{cubegraph}.\\

Let us now describe the inner and outer boundary paths of $H$. We set\\ $v_0=A, v_1=B, u=W$.
The inner boundary path consists of two sections: horizontal and vertical. For the horizontal
section we place horizontal edges, labeled by the numbers appearing in the first part of $\omega_H$
(according to the order in which
they appear in $\omega_H$. Note that the last vertex in the horizontal section is $W$).
For the vertical section we place vertical edges that are labeled by the numbers appearing in the second part of $\omega_H$.
Note that the definition of the $D_i$'s and the fact that $\C_{\J}$ is a circuit in $G_{k,n}$ implies that the inner boundary
path described above is indeed a subgraph of $G_{k,n}$. For the outer boundary path, consider the edges of $C_{\J}$
that are labeled by the numbers in $\widehat{\omega}$. Every such number appears in exactly one of
$\cup_{i=1}^k {D_i}^1$ and $\cup_{i=1}^k {D_i}^2$. Every edge that corresponds to the former set will be horizontal, and
every edge that corresponds to the latter set will be vertical (see the graph $Q_3$ in Figure~\ref{cubegraph} for an example).
Note that since $C_{\J}$ is a subgraph of $G_{k,n}$, then the outer boundary path is a subgraph of $G_{k,n}$ as well.
In addition, the inner and the outer boundary paths have the same number of vertical and horizontal edges.
Now, in order to show that the inner and the outer boundary paths described above induce a structure of oriented
Young graph, we need to show that the following holds:
\begin{enumerate}
  \item The first and the last edges in the outer boundary path are vertical and horizontal respectively.
  \item Once we establish the property above, we already know that the boundary of $H$ looks like the left part
  of Figure~\ref{gridiagram}. Let us now add internal horizontal and vertical edges
  (see the right part of Figures~\ref{gridiagram} and~\ref{cubegraph}),
  such that each horizontal edge is directed from left to right, and
  each vertical edge is directed from bottom to top.
  We label each horizontal edge by the same labeling as the horizontal edge from below
  in the inner path, and we label each vertical edge by the same labeling as the vertical edge from right
  in the inner path. The resulting graph is an oriented Young graph,
  and we need to show that this graph is a subgraph of $G_{k,n}$
  (we assume for now that $A \neq B$, and deal with the case $A=B$ later).
\end{enumerate}
\begin{figure}[h]
\centering
\includegraphics[height=1in]{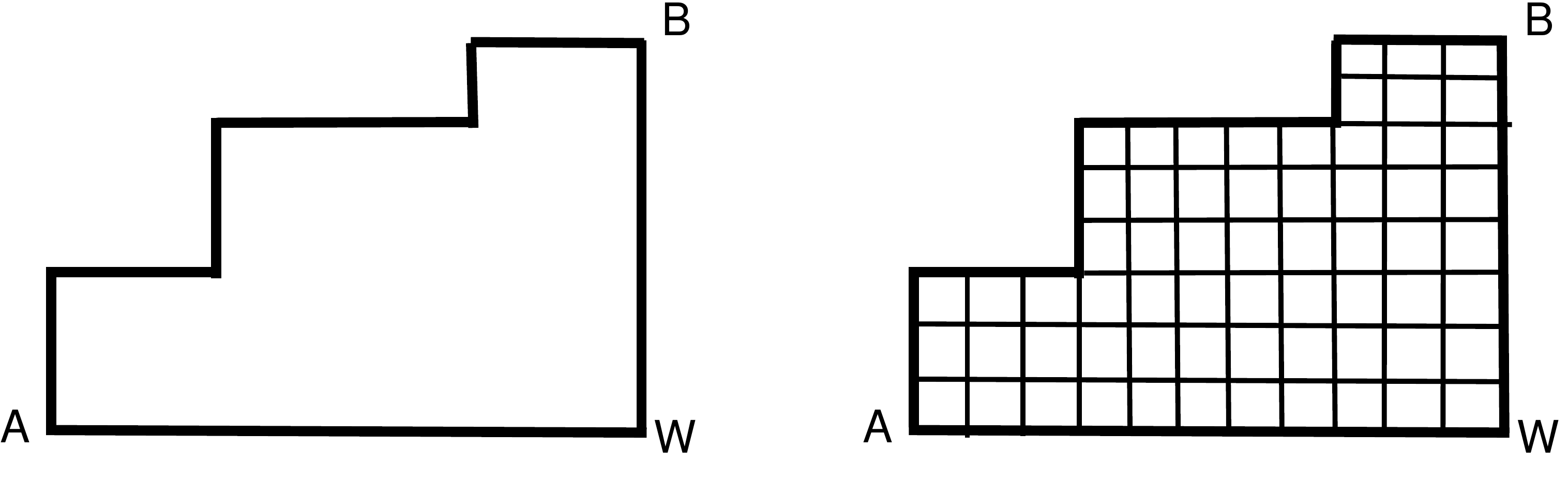}
\caption{}
\label{gridiagram}
\end{figure}
We start with (1). Assume for contradiction that the first edge is horizontal, and denote by $Z$ its other
vertex. Then $$Z=\{a_1,a_2,\ldots,a_{i-1},a_i+1,a_{i+1},\ldots,a_k\}$$ such that $a_i \in {D_i}^1$, which implies $a_i \leq d_i-1$.
Therefore, from (\ref{eqacb}) we have
$$a_1 \leq d_1 \leq a_2 \leq d_2 \ldots.... \leq a_{i-1} \leq d_{i-1} \leq a_{i}+1 \leq d_{i} \leq a_{i+1} \leq d_{i+1} \ldots \leq a_k \leq d_k.$$
This implies that $W$ is sorted with $Z$, and thus contradicts the fact that $W$ is not sorted with all
the vertices in the path $A \rightarrow B$ (excluding the endpoints) in $C_\J$.
We can similarly show that the last edge is
horizontal, so property (1) is established. We will prove property (2) by induction on the length of the
first part of $\omega_H$. If its length equals 1, then there exists an arc from $A$ to $W$
labeled by $a_j$ for some $1 \leq j \leq k$. Property (1) implies that
$\widehat{\omega}=\omega_1\omega_2\ldots\omega_{m-1}a_j$, and the situation is depicted in the left part of Figure~\ref{baseinduction}.
\begin{figure}[h]
\centering
\includegraphics[height=2in]{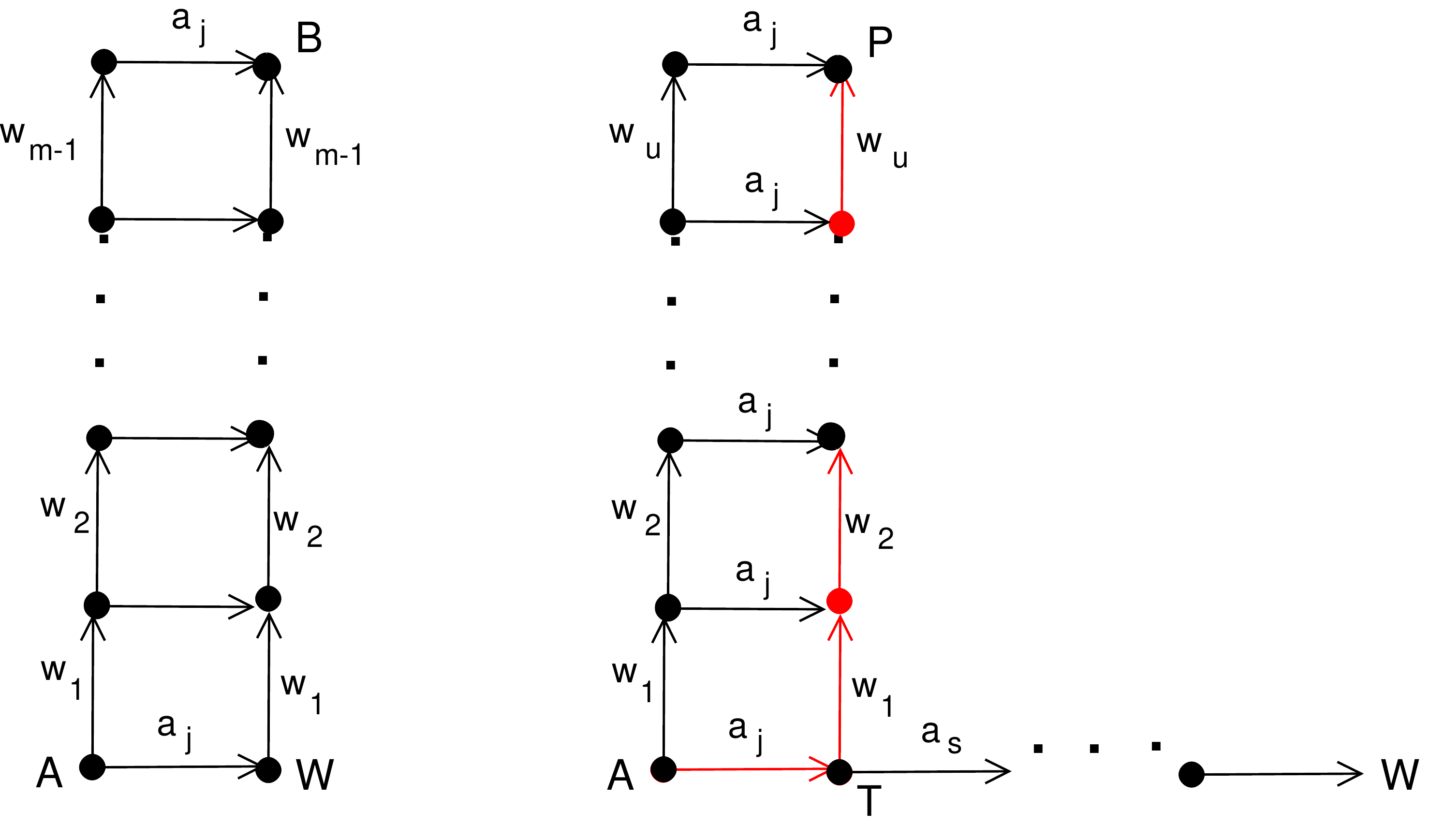}
\caption{}
\label{baseinduction}
\end{figure}
We need to show that by labeling every horizontal edge with $a_j$ we get a subgraph of $G_{k,n}$. Since $\widehat{\omega}$
is part of a permutation, $a_j \notin \{\omega_1,\omega_2,\ldots,\omega_{m-1}\}$. In addition we also have
$a_j+1 \notin \{\omega_1,\omega_2,\ldots,\omega_{m-1}\}$ (otherwise $a_j+1 \in A$, which contradicts the existence of the arc
from $A$ to $W$). Thus the base case is proven. Now assume that the length of the first part of $\omega_H$ equals $r>1$,
so the vertex that follows $A$ in the inner boundary path is of the form
$$T=\{a_1,a_2,\ldots,a_{j-1},a_j+1,a_{j+1},\ldots,a_k\}.$$ Then $\widehat{\omega}$ is of the form
$\widehat{\omega}=\omega_1\ldots\omega_{u}a_j\omega_{u+2}\ldots\omega_{m}$, and applying
the base case of the induction, we get the situation depicted in the right part of Figure~\ref{baseinduction}, where
$a_s \in \{\omega_{u+2},\ldots,\omega_{m}\}$ for all $s \in \{1,2,\ldots,k\}$ such that $s \neq j$. Now consider the minimal circuit $O$ that starts in $A$, continues along the red
path in the right part of Figure~\ref{baseinduction}, and then continues in the same way as $C_{\J}$.
The outer boundary path that corresponds to this circuit is associated with the following part of the permutation:
\begin{equation}\label{eqperm}
\omega_1\ldots\omega_{u}\omega_{u+2}\ldots\omega_{m}.
\end{equation}
The length of the corresponding first part is smaller than $r$, so we can use the inductive hypothesis and
construct the rest of the graph. To complete the proof we just need to verify that the initial vertical segment of the path
that corresponds to $O$ and starts in $A$ has at least $u$ edges. This follows from (\ref{eqperm}), so the case $A\neq B$ is done.
Now consider the case $A=B$. Recall that we order the elements in $B$ as follows: $B=\{a_2,a_3,\ldots,a_k,a_1\}$. Applying the
inductive process described above still leads us to an oriented Young graph. This graph is not a subgraph of $G_{k,n}$ (since
we duplicated one of its vertices, so we flatten the circuit), but we can still apply the reasoning from the beginning of
the proof and get the asserted claim.

The last paragraph of the proof will be dedicated to proving equation (\ref{eqacb}). If $A=B$ then this is trivial, so assume that
$A \neq B$. Denote the path $B \rightarrow A$ in $C_\J$
by $$P:= B \rightarrow T_1 \rightarrow T_2 \rightarrow \ldots \rightarrow T_r \rightarrow A$$ (so
it has $r+2$ vertices for some
$r \geq 0$). Since $W$ is sorted with all the elements in $P$, there exists a minimal circuit $C_{\pi}$ in $G_{k,n}$
that contains $W$ and all the vertices in $P$. We will show that $P$ is also a path in $C_{\pi}$.
Note that showing this will imply that $W$ is on the path from $A$ to $B$ in $C_{\pi}$, which by definition of minimal
circuits implies (\ref{eqacb}). We will start by showing that $B$ is followed by $T_1$ in $C_{\pi}$.
Since $B$ is followed by $T_1$ in $C_{\J}$, then $$T_1=\{b_1,\ldots,b_u,b_{u+1}+1,b_{u+2},\ldots,b_k\}$$ for some $u$ (the +1 is
modulo $n$). Now assume that a vertex $M \neq B$ is followed by $T_1$ in $C_{\pi}$. Then WLOG
$$M=\{b_1\ldots,b_{x-1},b_x-1,b_{x+1},\ldots,b_u,b_{u+1}+1,b_{u+2},\ldots,b_k\}.$$ Therefore,
$M$ and $B$ are not sorted, contradicting the fact that both of them are on $C_{\pi}$. We can show similarly that $T_i$ is followed
by $T_{i+1}$ for all $1 \leq i \leq r-1$ and that $T_{r+1}$ is followed by $A$, so (\ref{eqacb}) is proven.
\end{proof}

We conclude this part with the proof of Theorem~\ref{conjecturespecial}.
\begin{proof}
The case $t=2$ follows from Theorem~\ref{thm:kbynone}. Let us consider the case $k=2$, and let $W=\{a,b\}$. Since $\J$ is a maximal sorted set, there exists an element $A$ containing $a$ in $\J$, and similarly there exists an element $B$ containing $b$ in $\J$
(otherwise $\J$ would have at most $n-1$ elements). $W$ is sorted with both $A$ and $B$, so the claim
follows from Theorem~\ref{finaltheorem}. Finally, consider the case $t=3$. It is easy to verify the claim for $n \leq 5$,
so we assume $n \geq 6$. Using Claim~\ref{claimcubical} we obtain 8 cases
listed in Figure~\ref{cubical3}. In all the cases, the dotted lines represent the circuit that corresponds to $\J$
(the first and last end points might be the same point, similarly to the case $A=B$ in the proof of Theorem ~\ref{finaltheorem}),
and the red, blue and black edges correspond to vertices of cubical distance 1, 2 and 3 respectively. For the first 6 cases,
the claim follows from Lemma~\ref{zigzaglemmaone}, therefore we only need to consider the bottom 2 cases, starting from the left case
(labeled by 1). It is easy to verify that the labelings of the vertices are the one depicted in the figure
(there might be additional numbers, but they are common to all of the vertices so we can ignore them. Also, we have no assumption
on the order of $a,b,c$). Since $c \neq b+1$
(otherwise $W$ would not exist), let $Q$ be the point obtained by a detour $B \rightarrow Q \rightarrow D$ (such that the
edges are labeled by $c$ and $b$ respectively). Then $Q=\{a+1,b,c+1\}$. Now we can make the detour
$G \rightarrow W \rightarrow Q$ whose edges are labeled by $c$ and $a$ respectively.
Therefore $W$ is in fact of cubical distance at most 2 from $\J$, contradicting the assumption, so we are done with this case. Finally, consider the second case, depicted in the bottom right part of Figure~\ref{cubical3}.
Assume that $\Delta_I \leq 1$ for all $I \in {[n] \choose k}$ with equality iff $I \in \J$.
We have $\Delta_W\Delta_K<\Delta_P\Delta_{(a+1,b,c+1)}$, so $\Delta_W<\Delta_{(a+1,b,c+1)}$. In addition
$\Delta_T\Delta_{(a+1,b,c+1)}<\Delta_K\Delta_M$ which implies $\Delta_{(a+1,b,c+1)}<\Delta_M$. Therefore
$\Delta_W<\Delta_{(a+1,b,c+1)}<\Delta_M<1$, and we are done.
\end{proof}
\begin{figure}[h]
\centering
\includegraphics[height=5.5in]{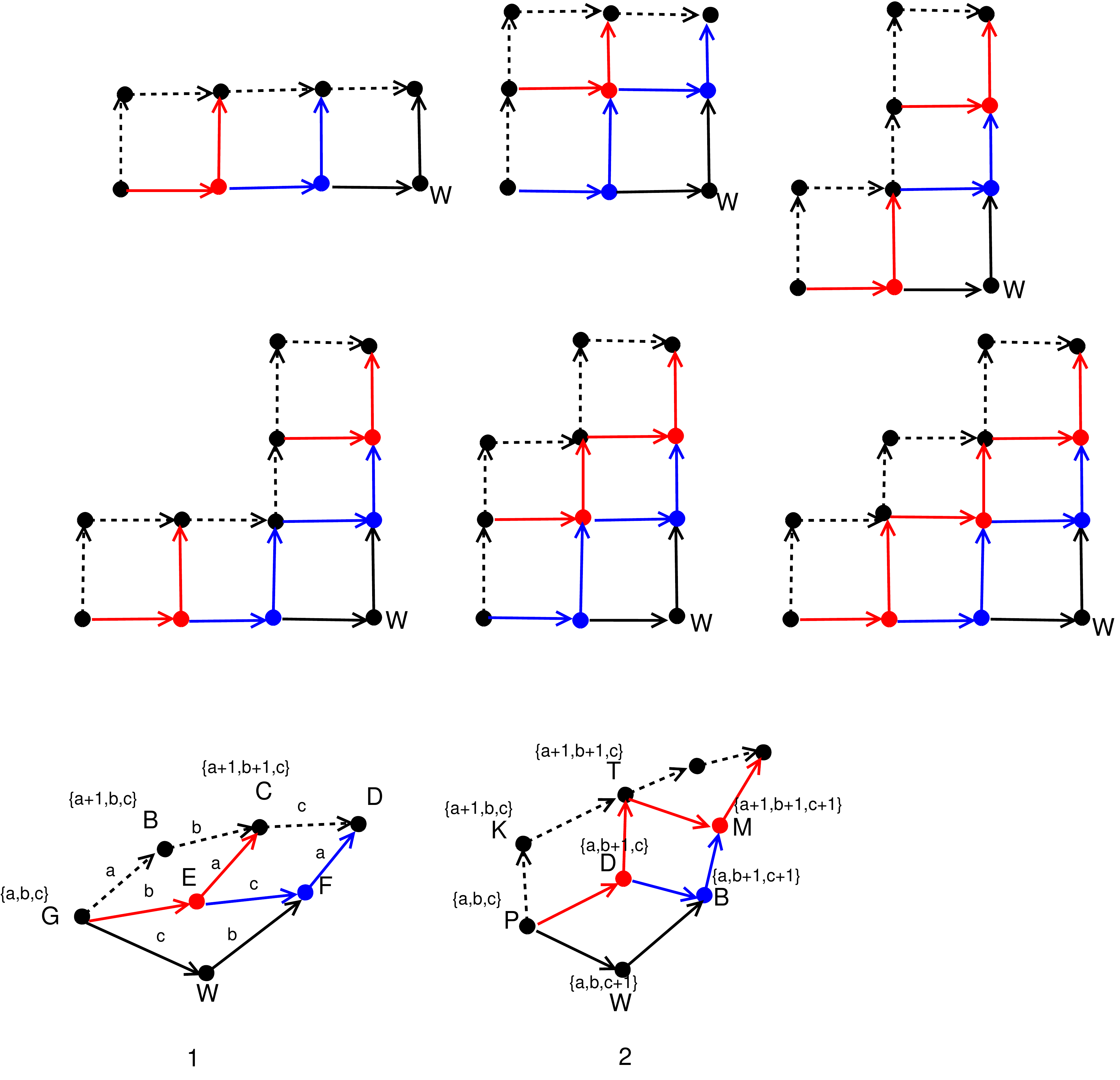}
\caption{}
\label{cubical3}
\end{figure}
Conjecture~\ref{conjecture1} deals with the case in which
$\mathcal{J}$ is maximal. We will now discuss the general case, in which $\mathcal{J}$ can be any sorted collection.
Theorem~\ref{thm:kbynone} implies that if $W \in {[n]\choose k}$ is a second largest minor, then $\epsilon_W$ is "close" to $\nabla_\J$. This notion of distance is formally defined in the following definition. This definition allows us to generalize this property for arrangements of $t^{th}$ largest minors ($t \geq 2$)
\begin{definition}\label{defdisten}
Let $r$ be an integer, $1 \leq i \leq j \leq n$, and denote by $H_{i,j,r}$ the affine hyperplane $\{x_i+x_{i+1}\cdots + x_j =r\}\subset\R^n$. Fix a point $x \in \R^n$. For $y \in \R^n$, we say that $H_{i,j,r}$ \emph{separates} $y$ from $x$ if one of the following holds:
\begin{itemize}
  \item $x$ and $y$ lie in the two disjoint halfspaces formed by $H_{i,j,r}$.
  \item $y$ lies on $H_{i,j,r}$ and $x$ does not.
\end{itemize}
Define \hspace{0.57in} $d_{ij}(x,y)=|\{r|\textrm{ the hyperplane } H_{i,j,r} \textrm{ separates } y \textrm{ from } x\}|$.\\
Finally, let \hspace{0.3in}  $B_{r}(x)=\{y\textrm{ }|\textrm{ }d_{ij}(x,y) \leq r \textrm{ for all } 1 \leq i \leq j \leq n \}$.
\end{definition}
For simplicity, we sometimes write $d_{ij}(I,J)$ instead of $d_{ij}(\epsilon_I,\epsilon_J)$. The notion $d_{ij}$ arises naturally in the discussion of sorted sets. In particular, by \cite[section 2.4]{polytopes},
$I$ and $J$ are sorted if and only if $d_{ij}(\epsilon_{I},\epsilon_{J})\leq 1$ for every $1 \leq i \leq j \leq n$.
\begin{theorem}\label{upperdistancebound}
Let $\mathcal{J}\subset{[n]\choose k}$ be some arrangement of largest minors, and let $\mathcal{Y}$ be a $(t,\mathcal{J})-$largest arrangement.
Then $\epsilon_Y \in B_{2^{t-1}}(J)$ for any $Y \in \mathcal{Y}, J \in \mathcal{J}$.
\end{theorem}
In order to prove this theorem, we will use the following claim which follows from the proof of Lemma 8.6 from\cite{main}. Here, and in the proof, we denote by $I',J'$ the sets $Sort_1(I,J),Sort_2(I,J)$.
\begin{claim}\label{claim:onplane}
Suppose that $\epsilon_I$ lies on $H_{i,j,\alpha}$ and $\epsilon_J$ lies on $H_{i,j,\beta}$. Then $\epsilon_{I'}$ and $\epsilon_{J'}$ lie on $H_{i,j, \lfloor \frac{\alpha+\beta}{2} \rfloor}$ and
$H_{i,j, \lceil \frac{\alpha+\beta}{2} \rceil}$ (not necessarily respectively).
\end{claim}

Let us now present the proof of Theorem~\ref{upperdistancebound}.
\begin{proof}
Fix some pair $1 \leq i \leq j \leq n$, $Y \in \mathcal{Y}, J \in \mathcal{J}$. We will prove the theorem by induction on $t$, starting with the case $t=2$.
If $Y$ and $J$ are sorted then $d_{ij}(Y,J) \leq 1$. If they are not sorted then by Theorem~\ref{thm:kbynone} $Y$ is sorted with some element $N$ such that $N$ is sorted with $J$. Therefore $d_{ij}(Y,J)\leq d_{ij}(Y,N)+d_{ij}(N,J) \leq 2$ so the claim is proven. Assume now that $t>2$, and that the claim is proven
for all the numbers up to $t-1$. Supposefor contradiction that $d_{ij}(Y,J)>2^{t-1}$ so $\epsilon_Y$ lies on $H_{i,j,\alpha}$
and $\epsilon_J$ lies on $H_{i,j,\beta}$ for some pair of numbers $\alpha,\beta$ that satisfy $|\alpha-\beta| > 2^{t-1}$. Hence using Skandera's inequalities we get
$\Delta_Y\Delta_J < \Delta_{Y'}\Delta_{J'}$. Assume WLOG that the maximal minors equal 1. Therefore, since $J \in \mathcal{J}$ we get
$\Delta_Y< \Delta_Y'\Delta_J'$. Recall that by Claim~\ref{claim:onplane} at least one of $Y',J'$ lies on $H_{i,j, \lceil \frac{\alpha+\beta}{2} \rceil}$, and assume that $\alpha>\beta$ (the other case can
be handled similarly). Then $\frac{\alpha+\beta}{2} > \frac{\beta+2^{t-1}+\beta}{2}=\beta+2^{t-2}$. Therefore
at least one of $d_{ij}(J,Y'), d_{ij}(J,J')$ is bigger than $2^{t-2}$, which by the inductive hypothesis implies that at least
one of $Y', J'$ is not a $(t-1,\mathcal{J})-$largest minor. Now, since we assumed that 1 is the largest minor, then
$\Delta_Y< \Delta_{J'}$ and $\Delta_Y< \Delta_{Y'}$, and hence $Y$ is not a $(t,\mathcal{J})-$largest minor, a contradiction.
\end{proof}
Thus, we get that if $W$ is a $(t,\mathcal{J})-$largest minor, then $W$ must lie within a ball of certain bounded radius around $\mathcal{J}$. We conclude this section with the following corollary.
\begin{corollary}
Let $\mathcal{Y}$ be an arrangement of $t^{th}$ largest minors, $t \geq 2$. Then all the elements $\epsilon_Y$, $Y\in \mathcal{Y}$ lie within a ball of radius $2^{t-1}$.
\end{corollary}


\FloatBarrier

\bibliographystyle{alpha}
\bibliography{paperbib}
\label{sec:biblio}

\end{document}